\documentclass[twoside,reqno]{amsart}

\usepackage{amsmath,amsfonts,amsthm}
\usepackage[colorlinks=true,linkcolor=blue,citecolor=blue,%
filecolor=blue,urlcolor=blue,pageanchor=true,plainpages=false,%
linktocpage]{hyperref}

\numberwithin{equation}{section}
\newtheorem{thm}{Theorem}[section]
\newtheorem{prop}[thm]{Proposition}
\newtheorem{cor}[thm]{Corollary}
\newtheorem{lem}[thm]{Lemma}
\theoremstyle{definition}
\newtheorem{defn}[thm]{Definition}
\newtheorem{rem}[thm]{Remark}

\newcommand{\N}{\mathbb{N}}
\newcommand{\Z}{\mathbb{Z}}
\newcommand{\R}{\mathbb{R}}

\newcommand{\cl}[1]{\overline{#1}}

\newcommand{\au}[1]{\textsc{#1}}
\newcommand{\titleart}[1]{\textrm{#1}}
\newcommand{\jour}[1]{\textit{#1}}
\newcommand{\volart}[1]{\textbf{#1}}
\newcommand{\no}[1]{\textit{no.} {#1}}

\renewcommand{\rho}{\varrho}
\renewcommand{\theta}{\vartheta}

\begin{document}


\title[A
Fredholm alternative for quasilinear elliptic equations]{A
Fredholm alternative for quasilinear elliptic equations
with right hand side measure}

\author{Michele Colturato}
\address{Dipartimento di Matematica ``F.~Casorati''\\
         Universit\`a di Pavia\\
         Via Ferrata 5\\
         27100 Pavia, Italy}
\email{michele.colturato01@universitadipavia.it}
\author{Marco Degiovanni}
\address{Dipartimento di Matematica e Fisica\\
         Universit\`a Cattolica del Sacro Cuore\\
         Via dei Musei 41\\
         25121 Bre\-scia, Italy}
\email{marco.degiovanni@unicatt.it}
\thanks{The second author is member of the 
        Gruppo Nazionale per l'Analisi Matematica, la 
				Probabilit\`a e le loro Applicazioni (GNAMPA) of the 
				Istituto Nazionale di Alta Matematica (INdAM)}
								
\keywords{Quasilinear elliptic equations, right hand side measure,
degree theory}

\subjclass[2010]{35J66, 35R06, 47H11}



%
\begin{abstract}
We consider a quasilinear elliptic equation, with right hand side
measure, which does not satisfy the usual coercivity assumption.
We prove an existence result in the line of the Fredholm
alternative.
For this purpose we develop a variant of degree theory suited
to this setting.
\end{abstract}
\maketitle 


\section{Introduction}
Let $\Omega$ be a bounded and open subset of $\R^N$ and let 
\[
a:\Omega\times\R^N\rightarrow\R^N\,,\qquad
b:\Omega\times(\R\times\R^N)\rightarrow\R
\]
be two Carath\'eodory functions.
We are interested in the existence of solutions $u$ to the problem
\begin{equation}
\label{eq:main}
\begin{cases}
- \mathrm{div}[a(x,\nabla u)] + b(x,u,\nabla u)=\mu 
&\qquad\text{in $\Omega$}\,,\\
u=0
&\qquad\text{on $\partial\Omega$}\,,
\end{cases}
\end{equation}
when $\mu$ is a Radon measure on $\Omega$ with bounded total
variation.
\par
If $b=0$ and 
$\left\{u\mapsto - \mathrm{div}[a(x,\nabla u)]\right\}$
is a coercive Leray-Lions operator from $W^{1,p}_0(\Omega)$ into
$W^{-1,p'}(\Omega)$, then the problem has 
been the object of several papers.
Let us mention in 
particular~\cite{boccardo_gallouet_orsina1996}, where the existence
and uniqueness of an entropy solution is proved when $\mu$ is
absolutely continuous with respect to the $p$-capacity, 
and~\cite{dalmaso_murat_orsina_prignet1999}, where the existence 
and stability of renormalized solutions is proved for a
general $\mu$.
When $\mu$ is absolutely continuous with respect to the 
$p$-capacity, the two concepts of entropy solution and
renormalized solution agree.
A sharper result is proved, in the case $p=N$, 
in~\cite{greco_iwaniec_sbordone1997},
while the case of $\mu$ absolutely continuous with respect to 
the $p$-capacity is treated also
in~\cite{kilpelainen_xu1996}, provided that $p>2-1/N$.
\par
If $p=2$ and the principal part is linear, then much more is known,
both in coercive and noncoercive situations (see
e.g.~\cite{amann_quittner1998, brezis_marcus_ponce2007,
degiovanni_scaglia2011, ferrero_saccon2006, ferrero_saccon2007,
ferrero_saccon2010, marcus_veron2014, orsina1993, veron2004}).
\par
On the other hand, if $\mu\in W^{-1,p'}(\Omega)$ and the operator
\[
\left\{u\mapsto - \mathrm{div}[a(x,\nabla u)]
+ b(x,u,\nabla u)\right\}
\] 
is well defined and continuous from $W^{1,p}_0(\Omega)$ into
$W^{-1,p'}(\Omega)$, then the problem has been treated, under 
various assumptions, also for $p\neq 2$ in the noncoercive case.
Let us mention, in particular, the results in the line of
the Fredholm alternative proved 
in~\cite{boccardo_drabek_giachetti_kucera1986} and then
developed with detailed descriptions, when resonance
occurs at the principal eigenvalue
(see~\cite{takac2010} and references therein).
\par
We are interested in a result in the same direction, when
$\mu$ is a Radon measure which has bounded total variation 
and is absolutely continuous with respect to the $p$-capacity.
More precisely, we assume that:
\begin{itemize}
\item[$(i)$]
there exist $1<p<\infty$, $\alpha_0, \alpha_2\in L^1(\Omega)$, 
$\alpha_1\in L^{p'}(\Omega)$, $\beta\in\R$ and $\nu>0$
such that
\begin{alignat*}{3}
&a(x,\xi)\cdot\xi \geq \nu |\xi|^p - \alpha_0(x)\,,\\
&|a(x,\xi)| \leq \alpha_1(x) + \beta|\xi|^{p-1} \,,\\
&|b(x,s,\xi)|\leq \alpha_2(x) + \beta|s|^{p-1} +\beta|\xi|^{p-1}\,,
\end{alignat*}
for a.e. $x\in\Omega$ and every $s\in\R$, $\xi\in\R^N$;
such a $p$ is clearly unique;
\item[$(ii)$]
we have
\[
[a(x,\xi)-a(x,\eta)]\cdot(\xi-\eta) >0
\]
for a.e. $x\in\Omega$ and every $\xi,\eta\in\R^N$ 
with $\xi\neq\eta$;
\item[$(iii)$]
there exist two Carath\'eodory functions
\[
a_{\infty}:\Omega\times\R^N\rightarrow\R^N\,,\qquad
b_{\infty}:\Omega\times(\R\times\R^N)\rightarrow\R\,,
\]
such that
\[
[a_{\infty}(x,\xi)-a_{\infty}(x,\eta)]\cdot(\xi-\eta) >0\,,
\]
for a.e. $x\in\Omega$ and every $\xi,\eta\in\R^N$ 
with $\xi\neq\eta$, and such that
\[
\lim_k \frac{a(x,\tau_k\xi_k)}{\tau_k^{p-1}} =
a_{\infty}(x,\xi)\,,\quad
\lim_k \frac{b(x,\tau_ks_k,\tau_k\xi_k)}{\tau_k^{p-1}} =
b_{\infty}(x,s,\xi)\,,
\]
for a.e. $x\in\Omega$,
whenever $\tau_k\to +\infty$, $s_k\to s$ and $\xi_k\to \xi$.
\end{itemize}
It easily follows that
\begin{alignat}{3}
\label{eq:ainftycoerc}
&a_{\infty}(x,\xi)\cdot\xi \geq \nu |\xi|^p\,,\\
\label{eq:ainftyest}
&|a_{\infty}(x,\xi)| \leq \beta|\xi|^{p-1} \,,\\
\label{eq:binftyest}
&|b_{\infty}(x,s,\xi)|\leq \beta|s|^{p-1} +\beta|\xi|^{p-1}\,,\\
\label{eq:ainftyhom}
&a_{\infty}(x,\tau\xi) = \tau^{p-1}\, a_{\infty}(x,\xi)\,,\\
\label{eq:binftyhom}
&b_{\infty}(x,\tau s,\tau\xi) 
= \tau^{p-1} \,b_{\infty}(x,s,\xi)\,,
\end{alignat}
for a.e. $x\in\Omega$ and every $\tau, s \in\R$, $\xi\in\R^N$
with $\tau\geq 0$.
\par
We aim to prove the next result.
\begin{thm}
\label{thm:main}
Under hypotheses $(i)$--$(iii)$, assume also that
\begin{equation}
\label{eq:abodd}
a_\infty(x,-\xi) = - a_\infty(x,\xi)\,,\qquad
b_\infty(x,-s,-\xi) = - b_\infty(x,s,\xi)\,,
\end{equation}
for a.e. $x\in \Omega$ and every $s\in\R$ and $\xi\in\R^N$.
\par
Then one at least of the following assertions is true:
\begin{itemize}
\item[$(a)$]
the problem
\[
\begin{cases}
u\in W^{1,p}_0(\Omega)\setminus\{0\}\,,\\
\noalign{\medskip}
- \mathrm{div}[a_{\infty}(x,\nabla u)] 
+ b_{\infty}(x,u,\nabla u)=0
&\qquad\text{in $W^{-1,p'}(\Omega)$}
\end{cases}
\]
admits a solution;
\item[$(b)$]
for every Radon measure $\mu$, which has bounded total variation 
and is absolutely continuous with respect to the $p$-capacity,
problem~\eqref{eq:main} admits an entropy solution $u$
(see the next Section~\ref{sect:degree} for the definition of 
entropy solution in this setting).
\end{itemize}
\end{thm}
In order to prove Theorem~\ref{thm:main}, we first develop an
adaptation of degree theory to our setting.
Since the problems with right hand side measure are
usually treated by an approximation procedure
involving operators of Leray-Lions type, 
we choose a development of the degree for maps of class~$(S)_+$ 
(see~\cite{browder1983, oregan_cho_chen2006, skrypnik1994}),
which naturally acts in the same setting.
\par
In Section~\ref{sect:entropy} we recall the main facts
concerning entropy solutions to problems of the form
\[
\begin{cases}
- \mathrm{div}[a(x,\nabla u)]=\mu 
&\qquad\text{in $\Omega$}\,,\\
u=0
&\qquad\text{on $\partial\Omega$}\,,
\end{cases}
\]
and prove some auxiliary results.
Then we introduce, in Section~\ref{sect:degree}, the adaptation 
of the topological degree and state its properties.
With this tool at hand, in Section~\ref{sect:main} we prove 
Theorem~\ref{thm:main}.
The subsequent sections are devoted to the construction of
such a degree and the proof of its properties.
\par
Since each Radon measure with bounded total variation
belongs to $W^{-1,p'}$ if $p>N$, we have no novelty in this case.
Therefore, from now on we only consider the case $1<p\leq N$.
\par
In the following, we denote by $\|~\|_p$ the usual $L^p$-norm 
and by $\|~\|_{-1,p'}$ the norm in $W^{-1,p'}(\Omega)$ dual to 
the norm $\|\nabla u\|_p$ in $W^{1,p}_0(\Omega)$.
If $s\in\R$, we set $s^{\pm}=\max\{\pm s,0\}$.


\section{Entropy solutions}
\label{sect:entropy}
From now on, $\Omega$ will denote a bounded and open 
subset of~$\R^N$ with $N\geq 2$.
According 
to~\cite{boccardo_gallouet_orsina1996}, if $1<p\leq N$ we denote 
by $\mathcal{M}_b^p(\Omega)$ the set of Radon measures $\mu$
on $\Omega$ whose total variation $|\mu|$ is bounded and absolutely
continuous with respect to the $p$-capacity.
\par
According 
to~\cite{benilan_boccardo_gallouet_gariepy_pierre_vazquez1995,
boccardo_gallouet_orsina1996}, we also denote by 
$\mathcal{T}^{1,p}_0(\Omega)$ the set of (classes of equivalence
of) functions $u:\Omega\rightarrow[-\infty,+\infty]$ such that
$|u|<+\infty$ a.e. in $\Omega$ and $T_k(u)\in W^{1,p}_0(\Omega)$ 
for any $k>0$, where
\[
T_k(s)=
\begin{cases}
s &\qquad\text{if $|s|\leq k$}\,,\\
k\,\dfrac{s}{|s|}
&\qquad\text{if $|s| > k$}\,.
\end{cases}
\]
If $u\in \mathcal{T}^{1,p}_0(\Omega)$, 
there exists one 
and only one measurable (class of equivalence)
$\nabla u:\Omega\rightarrow\R^N$ such that
$g(u)\in W^{1,p}_0(\Omega)$ and 
$\nabla[g(u)]=g'(u)\nabla u$ a.e. in $\Omega$,
whenever $g:\R\rightarrow\R$ is Lipschitz continuous with 
$g(0)=0$ and $g'(s)=0$ outside some compact subset 
of~$\R$.
\par
According to~\cite{dalmaso_murat_orsina_prignet1999}, any
$u\in\mathcal{T}^{1,p}_0(\Omega)$ has a Borel and
$\mathrm{cap}_p$-quasi continuous representative
$\tilde{u}:\Omega\rightarrow[-\infty,+\infty]$, defined up to 
a set of null $p$-capacity, which we still denote by $u$.
Of course, the set $\{|u|=+\infty\}$ has null measure, but could 
have positive $p$-capacity.
\par
If $u, u_1, u_2 \in \mathcal{T}^{1,p}_0(\Omega)$ and
$t\in\R$, it is easily seen that
\[
tu\,,\quad \max\{u_1,u_2\}\,,\quad 
\min\{u_1,u_2\}\in \mathcal{T}^{1,p}_0(\Omega)
\]
with
\begin{alignat*}{3}
&\nabla (tu) = t\nabla u\,,\\
&\text{$\nabla u_1 = \nabla u_2\quad$ a.e. in $\{u_1=u_2\}$}\,,\\
&\nabla \max\{u_1,u_2\} = 
\chi_{\{u_1>u_2\}}\nabla u_1 
+ \chi_{\{u_2\geq u_1\}}\nabla u_2\,,\\
&\nabla \min\{u_1,u_2\} = 
\chi_{\{u_1<u_2\}}\nabla u_1 
+ \chi_{\{u_2\leq u_1\}}\nabla u_2\,.
\end{alignat*}
On the contrary, one cannot say 
(see~\cite{benilan_boccardo_gallouet_gariepy_pierre_vazquez1995})
that $u_1 + u_2 \in \mathcal{T}^{1,p}_0(\Omega)$.
For reader's convenience, we also provide a proof 
of the next result.
\begin{prop}
\label{prop:nabla}
Let $u_1, u_2 \in \mathcal{T}^{1,p}_0(\Omega)$
with $\nabla u_1 = \nabla u_2$ a.e. in $\Omega$.
Assume also that
\[
\liminf_{k\to+\infty} \, \frac{1}{k^p}\,
\int_{\{k<|u_1|<2k\}} |\nabla u_1|^p\,dx 
< +\infty\,.
\]
\indent
Then $u_1 = u_2$ a.e. in $\Omega$.
\end{prop}
\begin{proof}
Let $\vartheta:\R\rightarrow[0,1]$ be a smooth function
with $\vartheta(s)=0$ for $|s|\leq 1$ and 
$\vartheta(s)=1$ for $|s|\geq 2$.
Then we have 
\[
T_{2k}(u_1) \,,\,\,T_1(u_2-T_{2k}(u_1)) \,,\,\,
\vartheta(u_1/k)
\in W^{1,p}_0(\Omega)\cap L^\infty(\Omega)\,,
\]
whence
\begin{multline*}
T_1(u_2-u_1)(1-\vartheta(u_1/k)) \\
= T_1(u_2-T_{2k}(u_1))(1-\vartheta(u_1/k)) \in 
W^{1,p}_0(\Omega)\cap L^\infty(\Omega)
\end{multline*}
with
\[
|\nabla[T_1(u_2-u_1)(1-\vartheta(u_1/k))]| \leq
\frac{\|\vartheta'\|_\infty}{k}\,\chi_{\{k<|u_1|<2k\}}\,
|\nabla u_1| \,.
\]
If $(k_j)$ is a sequence with $k_j\to+\infty$ and
\[
\sup_{j\in\N} \, \frac{1}{k_j^p}\,
\int_{\{k_j<|u_1|<2k_j\}} |\nabla u_1|^p\,dx 
< +\infty\,,
\]
it follows that
$(T_1(u_2-u_1)(1-\vartheta(u_1/k_j)))$ is convergent both to
$T_1(u_2-u_1)$ a.e. in~$\Omega$ and to $0$ weakly in 
$W^{1,p}_0(\Omega)$.
Then $T_1(u_2-u_1)=0$ a.e. in $\Omega$ and the assertion follows.
\end{proof}
\begin{rem}
Let $\Omega=\left\{x\in\R^N:\,\,|x|<1\right\}$ and
let $u_j(x)=v_j(|x|)$, where
\begin{alignat*}{3}
&v_1(r) &&=
\begin{cases}
\dfrac{1-r}{(1-2r)^2r}
&\quad\text{if $0<r<\dfrac{1}{2}$ or $\dfrac{1}{2}<r<1$}\,,\\
\noalign{\medskip}
+\infty
&\quad\text{if $r=0$ or $r=\dfrac{1}{2}$}\,.
\end{cases}
\\
&v_2(r) &&=
\begin{cases}
1 + \dfrac{1-r}{(1-2r)^2r}
&\quad\text{if $0<r<\dfrac{1}{2}$}\,,\\
\noalign{\medskip}
\dfrac{1-r}{(1-2r)^2r}
&\quad\text{if $\dfrac{1}{2}<r<1$}\,,\\
\noalign{\medskip}
+\infty
&\quad\text{if $r=0$ or $r=\dfrac{1}{2}$}\,.
\end{cases}
\end{alignat*}
Then $u_1, u_2 \in \mathcal{T}^{1,p}_0(\Omega)$
with $\nabla u_1 = \nabla u_2$ a.e. in $\Omega$,
but it is false that $u_1 = u_2$ a.e. in $\Omega$.
One can also observe that $-u_1\in \mathcal{T}^{1,p}_0(\Omega)$,
but $u_2+(-u_1)\not\in \mathcal{T}^{1,p}_0(\Omega)$.
\end{rem}
We are also interested in a smaller space, 
suggested by the techniques 
of~\cite{boccardo_gallouet1989, boccardo_gallouet1992, orsina1993}.
Let us denote by $\varphi_p:\R\rightarrow\R$ the increasing 
$C^\infty$-diffeomorphism such that
\[
\varphi_p'(s) = \frac{1}{\{(1+s^2)[\log(e+s^2)]^4
\}^{\frac{1}{2p}}}
\,,\qquad\varphi_p(0)=0\,.
\]
Then we denote by $\Phi^{1,p}_0(\Omega)$ the set of
(classes of equivalence of) functions $u:\Omega\rightarrow\R$ 
such that $\varphi_p(u)\in W^{1,p}_0(\Omega)$.
It is easily seen that 
\[
W^{1,p}_0(\Omega) \subseteq \Phi^{1,p}_0(\Omega)
\subseteq \mathcal{T}^{1,p}_0(\Omega)
\]
and that $\nabla[\varphi_p(u)] = \varphi_p'(u)\nabla u$
a.e. in $\Omega$, where $\nabla u$ has to be understood
in the sense of $\mathcal{T}^{1,p}_0(\Omega)$.
Moreover, any $u\in\Phi^{1,p}_0(\Omega)$ has a Borel and
$\mathrm{cap}_p$-quasi continuous representative
$\tilde{u}:\Omega\rightarrow\R$, defined up to 
a set of null $p$-capacity, which we still denote by $u$.
\par
Since $\left\{u\mapsto \varphi_p(u)\right\}$
is bijective from $\Phi^{1,p}_0(\Omega)$
onto $W^{1,p}_0(\Omega)$, there is a natural structure of
complete metric space on $\Phi^{1,p}_0(\Omega)$ which makes 
$\left\{u\mapsto \varphi_p(u)\right\}$
an isometry.
In particular, the distance function is given by
\[
d(u,v) = \|\nabla[\varphi_p(u)]-\nabla[\varphi_p(v)]\|_p
\qquad\text{for any $u,v \in \Phi^{1,p}_0(\Omega)$}\,.
\]
\begin{prop}
\label{prop:nablaphi}
Let $u_1 \in \Phi^{1,p}_0(\Omega)$ and 
$u_2 \in \mathcal{T}^{1,p}_0(\Omega)$
with $\nabla u_1 = \nabla u_2$ a.e. in $\Omega$.
Then $u_1 = u_2$ a.e. in $\Omega$.
\end{prop}
\begin{proof}
Taking into account the behavior of $\varphi_p'$ at
infinity, from
\[
\int_\Omega |\varphi_p'(u_1)|^p |\nabla u_1|^p\,dx <+\infty
\]
we infer that
\[
\int_\Omega \frac{|\nabla u_1|^p}{1+|u_1|^p}\,dx <+\infty\,.
\]
Since
\[
\frac{1}{k^p}\,
\int_{\{k<|u_1|<2k\}} |\nabla u_1|^p\,dx  \leq
2^{p+1}\,
\int_{\{k<|u_1|<2k\}} \frac{|\nabla u_1|^p}{1+|u_1|^p}\,dx
\]
for every $k\geq 1$, 
by Proposition~\ref{prop:nabla} the assertion follows.
\end{proof}
\par
Now let $a:\Omega\times\R^N\rightarrow\R^N$ be a Carath\'eodory 
function such that:
\begin{itemize}
\item[$(a_1)$]
there exist $1<p\leq N$, $\alpha_0\in L^1(\Omega)$, 
$\alpha_1\in L^{p'}(\Omega)$, $\beta_1\in\R$ and $\nu>0$
such that
\begin{alignat*}{3}
&a(x,\xi)\cdot\xi \geq \nu |\xi|^p - \alpha_0(x)\,,\\
&|a(x,\xi)| \leq \alpha_1(x) + \beta_1|\xi|^{p-1} \,,
\end{alignat*}
for a.e. $x\in\Omega$ and every $\xi\in\R^N$;
\item[$(a_2)$]
we have
\[
[a(x,\xi)-a(x,\eta)]\cdot(\xi-\eta) >0
\]
for a.e. $x\in\Omega$ and every $\xi,\eta\in\R^N$
with $\xi\neq\eta$.
\end{itemize}
\begin{defn}
Given $\mu\in \mathcal{M}_b^p(\Omega)$, we say that $u$ is an
\emph{entropy solution} of
\begin{equation}
\label{eq:mu}
\begin{cases}
- \mathrm{div}[a(x,\nabla u)] = \mu
&\qquad\text{in $\Omega$}\,,\\
u=0
&\qquad\text{on $\partial\Omega$}\,,
\end{cases}
\end{equation}
if $u\in \mathcal{T}^{1,p}_0(\Omega)$ and
\[
\int_\Omega a(x,\nabla u)\cdot\nabla [T_k(u - v)]\,dx
\leq \int_\Omega T_k(u - v)\,d\mu
\qquad\forall k>0\,,\,\, \forall v\in C^{\infty}_c(\Omega)\,.
\]
\end{defn}
If $u$ is an entropy solution of~\eqref{eq:mu}, then $u$ actually
satisfies the equality and for a much larger class of test 
functions.
Following the original idea of~\cite{brezis_browder1978},
we aim to prove a result in this direction.
\par
If $h, k\geq 0$, let $T_{h,k}:\R\rightarrow\R$ be the odd function 
such that
\[
T_{h,k}(s) =
\begin{cases}
0 &\qquad\text{if $0\leq s \leq h$}\,,\\
\noalign{\medskip}
s-h &\qquad\text{if $h < s < h+k$}\,,\\
\noalign{\medskip}
k &\qquad\text{if $s \geq h+k$}\,.
\end{cases}
\]
Then denote by 
$\widetilde{\mathcal{T}}^{1,p}_0(\Omega)$ the set of (classes of 
equivalence of) functions $u:\Omega\rightarrow[-\infty,+\infty]$ 
such that $|u|<+\infty$ a.e. in $\Omega$ and 
$T_{\varepsilon,k}(u)\in W^{1,p}_0(\Omega)$ 
whenever $\varepsilon>0$ and $k>0$.
It is easily seen that 
$\mathcal{T}^{1,p}_0(\Omega)\subseteq 
\widetilde{\mathcal{T}}^{1,p}_0(\Omega)$.
Moreover, if $u\in \widetilde{\mathcal{T}}^{1,p}_0(\Omega)$, 
there exists one 
and only one measurable (class of equivalence)
$\nabla u:\Omega\rightarrow\R^N$ such that
$\nabla u=0$ a.e. on $\left\{u=0\right\}$ and such that
$g(u)\in W^{1,p}_0(\Omega)$ with
$\nabla[g(u)]=g'(u)\nabla u$ a.e. in $\Omega$,
whenever $g:\R\rightarrow\R$ is Lipschitz continuous with 
$g(0)=0$ and $g'(s)=0$ outside some compact subset 
of~$]-\infty,0[\cup]0,+\infty[$.
If $u\in \mathcal{T}^{1,p}_0(\Omega)$, then the gradient of $u$ 
in the sense of $\mathcal{T}^{1,p}_0(\Omega)$ agrees with that
in the sense of $\widetilde{\mathcal{T}}^{1,p}_0(\Omega)$.
\par
As in the case of $\mathcal{T}^{1,p}_0(\Omega)$, any
$u\in\widetilde{\mathcal{T}}^{1,p}_0(\Omega)$ has a Borel and
$\mathrm{cap}_p$-quasi continuous representative
$\tilde{u}:\Omega\rightarrow[-\infty,+\infty]$, defined up to 
a set of null $p$-capacity, which we still denote by $u$.
Finally, let us point out that, if 
$u\in \widetilde{\mathcal{T}}^{1,p}_0(\Omega)$, we have 
$|u|^{t-1} u \in \widetilde{\mathcal{T}}^{1,p}_0(\Omega)$ 
whenever $t>0$.
\begin{thm}
\label{thm:bb}
Let $\mu\in \mathcal{M}_b^p(\Omega)$ and let $u$ be an 
entropy solution of~\eqref{eq:mu}.
Then we have
\begin{multline}
\label{eq:bb}
\int_\Omega (a(x,\nabla u)\cdot\nabla v)^+\,dx 
+ \int_\Omega (v \gamma)^-\,d|\mu| \\
= \int_\Omega (a(x,\nabla u)\cdot\nabla v)^-\,dx 
+ \int_\Omega (v \gamma)^+\,d|\mu|
\qquad\forall v\in \widetilde{\mathcal{T}}^{1,p}_0(\Omega)\,,
\end{multline}
where $d\mu = \gamma d|\mu|$ and $\gamma$ is a Borel function
with $|\gamma|=1$ $|\mu|$-a.e. in $\Omega$.
\end{thm}
\begin{proof}
As 
in~\cite[Lemma~3.3]{benilan_boccardo_gallouet_gariepy_pierre_vazquez1995},
we have
\begin{multline}
\label{eq:entrtest}
\int_\Omega a(x,\nabla u)\cdot\nabla [T_k(u - v)]\,dx
\leq \int_\Omega T_k(u - v)\,\gamma\,d|\mu| \\
\qquad\forall k>0\,,\,\, 
\forall v\in W^{1,p}_0(\Omega)\cap L^\infty(\Omega)\,.
\end{multline}
Now assume that 
$v\in W^{1,p}_0(\Omega)\cap L^\infty(\Omega)$
with $(a(x,\nabla u)\cdot\nabla v)^+\in L^1(\Omega)$
and take $k>\|v\|_\infty$.
By~\eqref{eq:entrtest} for every $t>0$ we have
\[
\int_{\Omega} a(x,\nabla u) \cdot \nabla [T_{tk}(u-tv)]\,dx
\leq
\int_{\Omega} T_{tk}(u-tv)\,\gamma\,d|\mu|\,.
\]
Since $T_{tk}(ts) = tT_{k}(s)$, it follows
\[
\int_{\Omega} a(x,\nabla u) \cdot 
\nabla \left[T_{k}\left(\frac{u}{t}-v\right)\right]\,dx
\leq
\int_{\Omega} T_{k}\left(\frac{u}{t} - v\right)\,\gamma\,d|\mu|\,,
\]
whence
\[
\begin{split}
\int_{\Omega} T_{k}\left(\frac{u}{t} - v\right)\,\gamma\,d|\mu| 
&\geq
\frac{1}{t}\,\int_{\left\{\left|\frac{u}{t}-v\right|<k\right\}} 
a(x,\nabla u) \cdot \nabla u\,dx \\
&\qquad\qquad\qquad
- \int_{\left\{\left|\frac{u}{t}-v\right|<k\right\}} 
a(x,\nabla u) \cdot \nabla v\,dx \\
&\geq
- \frac{1}{t}\,\int_{\Omega} \alpha_0\,dx -
\int_{\left\{\left|\frac{u}{t}-v\right|<k\right\}} 
a(x,\nabla u) \cdot \nabla v\,dx\,.
\end{split}
\]
Since $(a(x,\nabla u)\cdot\nabla v)^+\in L^1(\Omega)$, passing 
to the lower limit as $t\to+\infty$ and applying Fatou's lemma 
at the right hand side, we get
\[
\int_{\Omega} a(x,\nabla u) \cdot \nabla v\,dx
\geq
\int_{\Omega} v\,\gamma\,d|\mu|\,.
\]
It follows $a(x,\nabla u)\cdot\nabla v \in L^1(\Omega)$,
which allows to apply the same argument also to~$-v$,
obtaining
\[
\int_{\Omega} a(x,\nabla u) \cdot \nabla v\,dx
=
\int_{\Omega} v\,\gamma\,d|\mu|\,.
\]
\indent
Consider now 
$v\in \widetilde{\mathcal{T}}^{1,p}_0(\Omega)$
with $(a(x,\nabla u)\cdot\nabla v)^+\in L^1(\Omega)$
and $(v \gamma)^- \in L^1(\Omega,|\mu|)$.
Then $T_{1/k,k}(v)\in W^{1,p}_0(\Omega)\cap L^\infty(\Omega)$
with
\[
(a(x,\nabla u)\cdot\nabla T_{1/k,k}(v))^+ \leq
(a(x,\nabla u)\cdot\nabla v)^+\,,\qquad
(T_{1/k,k}(v) \gamma )^- \leq (v\gamma)^-\,.
\]
First of all, it follows
$a(x,\nabla u)\cdot\nabla T_{1/k,k}(v) \in L^1(\Omega)$ and
\[
\int_{\Omega} a(x,\nabla u) \cdot \nabla T_{1/k,k}(v)\,dx
=
\int_{\Omega} T_{1/k,k}(v) \,\gamma\,d|\mu|\,.
\]
Then, if $k\to \infty$, from Fatou's lemma we infer that
\[
\int_{\Omega} a(x,\nabla u) \cdot \nabla v\,dx
\geq
\int_{\Omega} v\,\gamma\,d|\mu|\,.
\]
It follows $a(x,\nabla u)\cdot\nabla v\in L^1(\Omega)$
and $v\,\gamma \in L^1(\Omega,|\mu|)$, so that we can argue on
$-v$, obtaining
\[
\int_{\Omega} a(x,\nabla u) \cdot \nabla v\,dx
=
\int_{\Omega} v\,\gamma\,d|\mu|\,.
\]
\indent
If $v\in \widetilde{\mathcal{T}}^{1,p}_0(\Omega)$ with 
$(a(x,\nabla u)\cdot\nabla v)^-\in L^1(\Omega)$
and $(v\,\gamma)^+ \in L^1(\Omega,|\mu|)$, the argument is 
analogous.
Otherwise, both sides of~\eqref{eq:bb} are $+\infty$.
\end{proof}
In the construction of the degree, a key role
will be played by the next regularity result.
\begin{thm}
\label{thm:regentr}
If $\mu\in \mathcal{M}_b^p(\Omega)$ and $u$ is an entropy 
solution of~\eqref{eq:mu}, then $u\in \Phi^{1,p}_0(\Omega)$.
In particular, the set $\{|u|=+\infty\}$ has null $p$-capacity.
Moreover, if we define an increasing and bounded 
$C^\infty$-function $\psi:\R\rightarrow\R$ by
\[
\psi'(s) = 
\frac{1}{\{(1+s^2)[\log(e+s^2)]^4
\}^{\frac{1}{2}}}
= (\varphi_p'(s))^p
\,,\qquad\psi(0)=0\,,
\]
then 
$\psi(u)\in W^{1,p}_0(\Omega)\cap L^{\infty}(\Omega)$,
$\psi'(u)\,a(x,\nabla u)\cdot \nabla u\in L^1(\Omega)$
and
\begin{gather}
\label{eq:sol}
\int_{\Omega} \psi'(u)\,a(x,\nabla u)\cdot \nabla u\,dx =
\int_\Omega \psi(u)\,d\mu\,,\\
\label{eq:estimate}
\nu \int_\Omega |\nabla [\varphi_p(u)]|^p\,dx \leq
\|\psi\|_\infty\,|\mu|(\Omega) + \|\alpha_0\|_1\,.
\end{gather}
\end{thm}
\begin{proof}
By Theorem~\ref{thm:bb} we have
\[
\int_\Omega a(x,\nabla u)\cdot \nabla[\psi(T_k(u))]\,dx =
\int_\Omega \psi(T_k(u))\,d\mu\,,
\]
whence
\[
\begin{split}
\nu \int_\Omega |\nabla[\varphi_p(T_k(u))]|^p\,dx
&=
\nu \int_{\{|u|<k\}} \psi'(u)\,|\nabla u|^p\,dx \\
&\leq
\int_{\{|u|<k\}} \psi'(u)\,
\left(a(x,\nabla u)\cdot \nabla u + \alpha_0\right)\,dx \\
&=
\int_\Omega \psi(T_k(u))\,d\mu 
+ \int_{\{|u|<k\}} \psi'(u)\alpha_0\,dx\,.
\end{split}
\]
Since $\psi'(s)\leq 1$, it follows
\[
\nu \int_\Omega |\nabla \varphi_p(T_k(u))|^p\,dx
\leq
\|\psi\|_\infty \, |\mu|(\Omega) + \|\alpha_0\|_1\,,
\]
so that $(\varphi_p(T_k(u)))$ is bounded in $W^{1,p}_0(\Omega)$.
Therefore $\varphi_p(u)\in W^{1,p}_0(\Omega)$ with
\[
\nu \int_\Omega |\nabla [\varphi_p(u)]|^p\,dx \leq
\|\psi\|_\infty\,|\mu|(\Omega) + \|\alpha_0\|_1\,.
\]
Since $\psi'(s)\leq \varphi_p'(s)$, 
\emph{a fortiori} we have that $(\psi(T_k(u)))$ 
is bounded in $W^{1,p}_0(\Omega)$, so that 
$\psi(u)\in W^{1,p}_0(\Omega)\cap L^\infty(\Omega)$.
Coming back to the equality
\begin{multline*}
\int_{\{|u|<k\}} \psi'(u)\,
\left(a(x,\nabla u)\cdot \nabla u + \alpha_0\right)\,dx \\
=
\int_\Omega \psi(T_k(u))\,d\mu 
+ \int_{\{|u|<k\}} \psi'(u)\alpha_0\,dx\,,
\end{multline*}
from the monotone convergence theorem we infer that
$\psi'(u)\,a(x,\nabla u)\cdot \nabla u \in L^1(\Omega)$
and
\[
\int_{\Omega} \psi'(u)\,a(x,\nabla u)\cdot \nabla u\,dx 
= \int_\Omega \psi(u)\,d\mu \,.
\]
\end{proof}
\begin{rem}
\label{rem:equiv}
By Theorem~\ref{thm:regentr}, in the definition of
entropy solution it is equivalent to require
$u\in \mathcal{T}^{1,p}_0(\Omega)$ or
$u\in \Phi^{1,p}_0(\Omega)$.
\end{rem}
Now let us recall the main result on entropy solutions.
\begin{thm}
\label{thm:mu}
For every $\mu\in \mathcal{M}_b^p(\Omega)$, there exists one 
and only one entropy solution~$u$ of~\eqref{eq:mu}.
Moreover, if $\mu_1, \mu_2\in \mathcal{M}_b^p(\Omega)$ satisfy
\[
\int_\Omega v\,d\mu_1 \leq \int_\Omega v\,d\mu_2
\qquad\text{for any $v\in C^{\infty}_c(\Omega)$ with $v\geq 0$}
\]
and $u_1, u_2\in \Phi^{1,p}_0(\Omega)$ are the corresponding
entropy solutions of~\eqref{eq:mu},
it follows $u_1\leq u_2$ a.e. in $\Omega$.
\end{thm}
\begin{proof}
If $u$ is an entropy solution of~\eqref{eq:mu}, we clearly have
\[
\lim_h \int_{\{|u|\geq h\}} \alpha_0\,dx = 0\,.
\]
As 
in~\cite[Formula~(7)]{boccardo_gallouet_orsina1996}, it follows
\begin{equation}
\label{eq:hk}
\lim_h \int_{\{h\leq |u|\leq h+k\}} |\nabla u|^p\,dx = 0
\qquad\text{for any $k>0$}\,.
\end{equation}
Then the existence and uniqueness of the entropy solution $u$ 
can be proved as in~\cite{boccardo_gallouet_orsina1996}
(see 
also~\cite{benilan_boccardo_gallouet_gariepy_pierre_vazquez1995}
and the proof of the next Lemma~\ref{lem:cont} for the existence 
part).
More specifically, the order preserving can be proved as 
in~\cite[Theorem~2.5]{kilpelainen_xu1996}, where the condition
$p>2-1/N$ is assumed.
However, the same argument works in our case.
We sketch it for reader's convenience.
\par
By Theorem~\ref{thm:regentr} we have 
$u_1, u_2\in \Phi^{1,p}_0(\Omega)$.
Moreover, by Theorem~\ref{thm:bb}, it holds whenever $0<k<h$
\[
\begin{split}
\int_\Omega [T_k(u_1 - T_h(u_2))]^+\,d\mu_1 &=
\int_\Omega a(x,\nabla u_1)\cdot
\nabla [T_k(u_1 - T_h(u_2))]^+\,dx \\
&\geq
\int_{\{|u_1|<h\,,\,\,|u_2|<h\,,\,\,0<u_1-u_2<k\}} 
a(x,\nabla u_1)\cdot\nabla(u_1 - u_2)\,dx \\
&\quad
- \int_{\{h\leq |u_1| < h+k\,,\,\,h-k<|u_2|<h\}} 
|a(x,\nabla u_1)|\,|\nabla u_2|\,dx \\
&\quad\quad
- \int_{\{|u_1|\geq h\}} 
\alpha_0\,dx 
- \int_{\{|u_2|\geq h\}} 
\alpha_0\,dx \,,
\end{split}
\]
\[
\begin{split}
\int_\Omega [T_k(u_2 - T_h(u_1))]^-\,d\mu_2 &=
\int_\Omega a(x,\nabla u_2)\cdot
\nabla [T_k(u_2 - T_h(u_1))]^-\,dx \\
&\leq
\int_{\{|u_1|<h\,,\,\,|u_2|<h\,,\,\,0<u_1-u_2<k\}} 
a(x,\nabla u_2)\cdot\nabla(u_1 - u_2)\,dx \\
&\quad
+ \int_{\{h\leq |u_2| < h+k\,,\,\,h-k<|u_1|<h\}} 
|a(x,\nabla u_2)|\,|\nabla u_1|\,dx \\
&\quad\quad
+ \int_{\{|u_2|\geq h\}} 
\alpha_0\,dx 
+ \int_{\{|u_1|\geq h\}} 
\alpha_0\,dx\,.
\end{split}
\]
It follows
\begin{multline*}
\int_\Omega [T_k(u_1 - T_h(u_2))]^+\,d\mu_1 
- \int_\Omega [T_k(u_2 - T_h(u_1))]^-\,d\mu_2 \\
\geq
\int_{\{|u_1|<h\,,\,\,|u_2|<h\,,\,\,0<u_1-u_2<k\}} 
[a(x,\nabla u_1)-a(x,\nabla u_2)]\cdot\nabla(u_1 - u_2)\,dx \\
- \int_{\{h\leq |u_1| < h+k\,,\,\,h-k<|u_2|<h\}} 
|a(x,\nabla u_1)|\,|\nabla u_2|\,dx \\
\null\qquad\qquad\qquad\qquad
- \int_{\{h\leq |u_2| < h+k\,,\,\,h-k<|u_1|<h\}} 
|a(x,\nabla u_2)|\,|\nabla u_1|\,dx \\
- 2 \int_{\{|u_1|\geq h\}} 
\alpha_0\,dx 
- 2 \int_{\{|u_2|\geq h\}} 
\alpha_0\,dx\,.
\end{multline*}
Passing to the limit as $h\to+\infty$ and taking into
account~\eqref{eq:hk}, we get
\begin{multline*}
0\geq 
- \int_{\Omega} [T_k(u_1 - u_2)]^+\,d(\mu_2-\mu_1) \\
\geq
\int_{\{0<u_1-u_2<k\}} 
[a(x,\nabla u_1)-a(x,\nabla u_2)]\cdot\nabla(u_1 - u_2)\,dx\,,
\end{multline*}
whence $\nabla u_1=\nabla u_2$ a.e. in $\{u_1>u_2\}$, namely
$\nabla[\max\{u_1,u_2\}] = \nabla u_2$ a.e. in~$\Omega$.
Since $\max\{u_1,u_2\}\in\mathcal{T}^{1,p}_0(\Omega)$ and 
$u_2\in\Phi^{1,p}_0(\Omega)$, by Proposition~\ref{prop:nablaphi}
we infer that $\max\{u_1,u_2\}=u_2$, namely $u_1\leq u_2$.
\end{proof}
\begin{prop}
\label{prop:Phi}
The following facts hold:
\begin{itemize}
\item[$(a)$]
if $u\in \Phi^{1,p}_0(\Omega)$, then 
$|\nabla u|^{p-1}\in L^q(\Omega)$ and
$|u|^{p-1} \in L^r(\Omega)$,
whenever $q<\frac{N}{N-1}$ and $r<\frac{N}{N-p}$
($r<\infty$ if $p=N$);
\item[$(b)$]
if $(u_n)$ is bounded in $\Phi^{1,p}_0(\Omega)$,
then $(|\nabla u_n|^{p-1})$ is bounded in $L^q(\Omega)$,
whenever $q<\frac{N}{N-1}$;
moreover, there exists $u\in\Phi^{1,p}_0(\Omega)$ such that, 
up to a subsequence, $(|u_n|^{p-2}\,u_n)$ is strongly convergent 
to $|u|^{p-2}\,u$  in $L^r(\Omega)$, whenever $r<\frac{N}{N-p}$;
\item[$(c)$]
if $u_n, u \in \Phi^{1,p}_0(\Omega)$,
$(u_n)$ is bounded in $\Phi^{1,p}_0(\Omega)$ and
$\nabla u_n \to \nabla u$ a.e. in~$\Omega$, then we have: 
\begin{multline*}
\null\qquad\qquad
\lim_n |\nabla u_n|^{p-2}\,\nabla u_n 
= |\nabla u|^{p-2}\,\nabla u \\
\text{strongly in $L^q(\Omega;\R^N)$, 
for any $q<\frac{N}{N-1}$}\,,
\end{multline*}
\item[$(d)$]
if $(u_n)$ is convergent to $u$ in 
$\Phi^{1,p}_0(\Omega)$, then we have
\[
\lim_n T_k(u_n) = T_k(u) 
\qquad\text{strongly in $W^{1,p}_0(\Omega)$, for any $k>0$}\,.
\]
\end{itemize}
\end{prop}
\begin{proof}
The argument is an adaptation of the techniques 
of~\cite{boccardo_gallouet1989, boccardo_gallouet1992, orsina1993}.
If $u\in \Phi^{1,p}_0(\Omega)$ and $p<N$, we have 
$\varphi_p(u)\in L^{p^*}(\Omega)$.
Since $\varphi_p'(s)$ behaves like
\[
\frac{1}{\{|s|(\log |s|)^2\}^{\frac{1}{p}}}
\]
at infinity, it follows that $|u|^{p-1} \in L^r(\Omega)$ 
whenever $r<\frac{N}{N-p}$.
\par
If $q<\frac{p}{p-1}$, we also have
\[
\begin{split}
\int_\Omega |\nabla u|^{(p-1)q}\,dx &=
\int_\Omega |\nabla [\varphi_p(u)]|^{(p-1)q}\,
\frac{1}{\varphi_p'(u)^{(p-1)q}}\,dx \\
&\leq
\left(\int_\Omega |\nabla [\varphi_p(u)]|^p\,dx
\right)^{\frac{(p-1)q}{p}}
\left(\int_\Omega 
\frac{1}{\varphi_p'(u)^{\frac{p(p-1)q}{p-(p-1)q}}}\,dx
\right)^{\frac{p-(p-1)q}{p}}\,.
\end{split}
\]
In particular, if $q<\frac{N}{N-1}$ we also have
\[
\frac{p(p-1)q}{p-(p-1)q} < p(p-1)\,\frac{N}{N-p}\,.
\]
Taking into account the behavior of $\varphi_p'$ 
at infinity and the previous assertion, we infer that 
$|\nabla u|^{p-1}\in L^q(\Omega)$.
Therefore assertion~$(a)$ is proved.
\par
The same argument shows that,
if $(u_n)$ is bounded in $\Phi^{1,p}_0(\Omega)$, then
$(|\nabla u_n|^{p-1})$ is bounded in $L^q(\Omega)$
and $(|u_n|^{p-1})$ is bounded in $L^r(\Omega)$,
whenever $q<\frac{N}{N-1}$ and $r<\frac{N}{N-p}$.
Moreover, up to a subsequence, $(\varphi_p(u_n))$
is convergent to $\varphi_p(u)$ weakly in $W^{1,p}_0(\Omega)$
and a.e. in $\Omega$.
It follows that $(u_n)$ is convergent to $u$ a.e. in $\Omega$,
so that 
$(|u_n|^{p-2}\,u_n)$ is strongly convergent to $|u|^{p-2}\,u$ 
in $L^r(\Omega)$, whenever $r<\frac{N}{N-p}$.
Therefore, assertion $(b)$ also holds.
Then~$(c)$ easily follows.
\par
Finally, if $(u_n)$ is convergent to $u$ in 
$\Phi^{1,p}_0(\Omega)$, then 
$(\varphi_p(T_k(u_n)))$ is strongly convergent to 
$\varphi_p(T_k(u))$ in $W^{1,p}_0(\Omega)$,
so that $(T_k(u_n))$ is strongly convergent to 
$T_k(u)$ in $W^{1,p}_0(\Omega)$.
\par
If $p=N$, the arguments are similar.
\end{proof}
Finally, up to minor variants due to the presence of $\alpha_0$
in assumption~$(a_1)$, the next regularity result can
be proved as
in~\cite{boccardo_gallouet1989, boccardo_gallouet1992, orsina1993}.
\begin{thm}
\label{thm:reglp}
If $\mu\in \mathcal{M}_b^p(\Omega)$ and $u$ is the entropy 
solution of~\eqref{eq:mu}, then the following facts hold:
\begin{itemize}
\item[$(a)$]
if $\mu\in L^m(\Omega)$ with $1<m<(p^*)'$, we have
$|\nabla u|^{p-1}\in L^{m^*}(\Omega)$ and 
$|u|^{p-1}\in L^{\frac{Nm}{N-pm}}(\Omega)$;
\item[$(b)$]
if $\mu\in \mathcal{M}_b^p(\Omega)\cap W^{-1,p'}(\Omega)$, 
we have $u\in W^{1,p}_0(\Omega)$ and 
\[
- \mathrm{div}[a(x,\nabla u)] = \mu
\qquad\text{in $W^{-1,p'}(\Omega)$}\,.
\]
\end{itemize}
\end{thm}
%


\section{A degree for a class of quasilinear elliptic equations}
\label{sect:degree}
Consider again a bounded and open subset $\Omega$ of $\R^N$,
a Carath\'eodory function $a:\Omega\times\R^N\rightarrow\R^N$ 
satisfying $(a_1)$ and $(a_2)$ and $\mu\in\mathcal{M}_b^p(\Omega)$.
Let also 
\[
b:\Omega\times(\R\times\R^N)\rightarrow\R
\]
be a Carath\'eodory function such that:
\begin{itemize}
\item[$(a_3)$]
there exist $\alpha_2\in L^1(\Omega)$, $\beta_2\in\R$,
$0<q<\frac{N(p-1)}{N-1}$ and $0<r<\frac{N(p-1)}{N-p}$
($0<r<+\infty$ if $p=N$) such that
\[
|b(x,s,\xi)|\leq \alpha_2(x) + \beta_2|s|^r +\beta_2|\xi|^q
\]
for a.e. $x\in \Omega$ and every $s\in\R$ and $\xi\in\R^N$.
\end{itemize}
By Proposition~\ref{prop:Phi} and $(a_3)$, we have
$b(x,u,\nabla u)\in L^1(\Omega)$ for any 
$u\in \Phi^{1,p}_0(\Omega)$.
According to Remark~\ref{rem:equiv}, we say that
$u$ is an \emph{entropy solution} of
\begin{equation}
\label{eq:bmu}
\begin{cases}
- \mathrm{div}[a(x,\nabla u)] + b(x,u,\nabla u)=\mu 
&\qquad\text{in $\Omega$}\,,\\
u=0
&\qquad\text{on $\partial\Omega$}\,,
\end{cases}
\end{equation}
if $u\in \Phi^{1,p}_0(\Omega)$ and
\begin{multline*}
\int_\Omega a(x,\nabla u)\cdot\nabla [T_k(u - v)]\,dx
+ \int_\Omega b(x,u,\nabla u)\,T_k(u - v)\,dx \\
\leq \int_\Omega T_k(u-v)\,d\mu
\qquad\forall k>0\,,\,\, \forall v\in C^{\infty}_c(\Omega)\,.
\end{multline*}
\begin{rem}
\label{rem:dec}
Let $u\in \Phi^{1,p}_0(\Omega)$
and let also $\hat{\mu}\in\mathcal{M}_b^p(\Omega)$ and
\[
\hat{a}:\Omega\times\R^N\rightarrow\R^N\,,\qquad
\hat{b}:\Omega\times(\R\times\R^N)\rightarrow\R
\]
be two Carath\'eodory functions satisfying $(a_1)$ -- $(a_3)$.
Assume that
\begin{multline*}
\int_{\Omega} \{a(x,\nabla z)\cdot\nabla v + b(x,z,\nabla z)v\}\,dx
- \int_\Omega v\,d\mu \\
= \int_{\Omega} \{\hat{a}(x,\nabla z)\cdot\nabla v 
+ \hat{b}(x,z,\nabla z)v\}\,dx 
- \int_\Omega v\,d\hat{\mu}  \\
\qquad \forall z, v\in W^{1,p}_0(\Omega)\cap L^{\infty}(\Omega)\,.
\end{multline*}
\par
Then $u$ is an entropy solution of~\eqref{eq:bmu} if and only if
$u$ is an entropy solution of
\[
\begin{cases}
- \mathrm{div}[\hat{a}(x,\nabla u)] 
+ \hat{b}(x,u,\nabla u)=\hat{\mu} 
&\qquad\text{in $\Omega$}\,,\\
u=0
&\qquad\text{on $\partial\Omega$}\,.
\end{cases}
\]
\end{rem}
\begin{proof}
For every $u\in \Phi^{1,p}_0(\Omega)$,
$v\in C^\infty_c(\Omega)$ and $h,k>0$, we have
\begin{multline*}
\int_{\Omega} \{a(x,\nabla[T_h(u)])\cdot\nabla [T_k(u-v)] 
+ b(x,T_h(u),\nabla [T_h(u)])T_k(u-v)\}\,dx \\
\null\qquad\qquad\qquad\qquad\qquad\qquad
\qquad\qquad\qquad\qquad\qquad
- \int_\Omega T_k(u-v)\,d\mu \\
= \int_{\Omega} \{\hat{a}(x,\nabla[T_h(u)])\cdot\nabla [T_k(u-v)] 
+ \hat{b}(x,T_h(u),\nabla[T_h(u)])T_k(u-v)\}\,dx \\ 
- \int_\Omega T_k(u-v)\,d\hat{\mu}  \,.
\end{multline*}
Passing to the limit as $h\to+\infty$, we get
\begin{multline*}
\int_{\Omega} \{a(x,\nabla u)\cdot\nabla [T_k(u-v)] 
+ b(x,u,\nabla u)T_k(u-v)\}\,dx 
- \int_\Omega T_k(u-v)\,d\mu \\
= \int_{\Omega} \{\hat{a}(x,\nabla u)\cdot\nabla [T_k(u-v)] 
+ \hat{b}(x,u,\nabla u)T_k(u-v)\}\,dx 
- \int_\Omega T_k(u-v)\,d\hat{\mu}  
\end{multline*}
and the assertion follows.
\end{proof}
We will also consider parametric problems, 
in which $T$ is a metrizable topological space and
\[
a:\Omega\times(\R^N\times T)\rightarrow \R^N\,,\qquad
b:\Omega\times(\R\times\R^N\times T)\rightarrow \R
\]
are two Carath\'eodory functions satisfying $(a_1)$ -- $(a_3)$
uniformly, namely:
\begin{itemize}
\item[$(u_1)$]
there exist $1<p\leq N$, $\alpha_0\in L^1(\Omega)$, 
$\alpha_1\in L^{p'}(\Omega)$, $\beta_1\in\R$ and $\nu>0$
such that
\begin{alignat*}{3}
&a_t(x,\xi)\cdot\xi \geq \nu |\xi|^p - \alpha_0(x)\,,\\
&|a_t(x,\xi)| \leq \alpha_1(x) + \beta_1|\xi|^{p-1} \,,
\end{alignat*}
for a.e. $x\in\Omega$ and every $\xi\in\R^N$ and $t\in T$;
\item[$(u_2)$]
we have
\[
[a_t(x,\xi)-a_t(x,\eta)]\cdot(\xi-\eta) >0
\]
for a.e. $x\in\Omega$ and every $\xi,\eta\in\R^N$ and
$t\in T$ with $\xi\neq\eta$;
\item[$(u_3)$]
there exist $\alpha_2\in L^1(\Omega)$, $\beta_2\in\R$,
$0<q<\frac{N(p-1)}{N-1}$ and $0<r<\frac{N(p-1)}{N-p}$
such that
\[
|b_t(x,s,\xi)|\leq \alpha_2(x) + \beta_2|s|^r +\beta_2|\xi|^q
\]
for a.e. $x\in \Omega$ and every $s\in\R$, $\xi\in\R^N$
and $t\in T$
(we write $a_t(x,\xi)$, $b_t(x,s,\xi)$ instead of 
$a(x,(\xi,t))$, $b(x,(s,\xi,t))$).
\end{itemize}
In Section~\ref{sect:degreeeq} we will see that it is possible
to define a topological degree
\[
\mathrm{deg}(- \mathrm{div}[a(x,\nabla u)] 
+ b(x,u,\nabla u),U,\mu)\in \Z
\]
whenever $U$ is a bounded and open subset of $\Phi^{1,p}_0(\Omega)$
such that~\eqref{eq:bmu} has no entropy solution $u\in\partial U$.
We state here the main properties, referring to 
Section~\ref{sect:degreeeq} for the proofs and further details.
\begin{thm}
\label{thm:consistency}
\textbf{\emph{(Consistency property)}}
Suppose that 
$\mu\in\mathcal{M}_b^p(\Omega)\cap W^{-1,p'}(\Omega)$ and that
$\alpha_2\in L^1(\Omega)\cap W^{-1,p'}(\Omega)$ in 
assumption~$(a_3)$.
\par
Then the following facts hold:
\begin{itemize}
\item[$(a)$]
we have
\[
\begin{cases}
b(x,u,\nabla u)v\in L^1(\Omega) \\
\noalign{\medskip}
b(x,u,\nabla u)\in L^1(\Omega)\cap W^{-1,p'}(\Omega) 
\end{cases}
\qquad\text{for any $u,v\in W^{1,p}_0(\Omega)$}
\]
and the map
\[
\begin{array}{ccc}
W^{1,p}_0(\Omega) & \longrightarrow & W^{-1,p'}(\Omega) \\
\noalign{\medskip}
u & \mapsto & -\mathrm{div}[a(x,\nabla u)]+b(x,u,\nabla u)
\end{array}
\]
is continuous and of class~$(S)_+$;
\item[$(b)$]
every entropy solution of~\eqref{eq:bmu} belongs to 
$W^{1,p}_0(\Omega)$ and every $u\in W^{1,p}_0(\Omega)$
is an entropy solution of~\eqref{eq:bmu} if and only if
\[
- \mathrm{div}[a(x,\nabla u)] + b(x,u,\nabla u)=\mu
\qquad\text{in $W^{-1,p'}(\Omega)$}\,;
\]
\item[$(c)$]
if $U$ is a bounded and open subset of $\Phi^{1,p}_0(\Omega)$
such that~\eqref{eq:bmu} has no entropy solution  
$u\in\partial U$, then the set
\[
\left\{u\in U:\,\,
- \mathrm{div}[a(x,\nabla u)] + b(x,u,\nabla u)=\mu\right\}
\]
is compact in $W^{1,p}_0(\Omega)$ and we have
\begin{multline*}
\null\qquad\qquad
\mathrm{deg}(- \mathrm{div}[a(x,\nabla u)] 
+ b(x,u,\nabla u),U,\mu) \\
= \mathrm{deg}_{(S)_+}(- \mathrm{div}[a(x,\nabla u)] 
+ b(x,u,\nabla u),U\cap V,\mu)\,,
\end{multline*}
whenever $V$ is a bounded and open subset of $W^{1,p}_0(\Omega)$
such that there are no solutions of~\eqref{eq:bmu}
in $U\setminus V$ (we have denoted by $\mathrm{deg}_{(S)_+}$ 
the degree for maps of class $(S)_+$ as defined
in~\cite{browder1983, oregan_cho_chen2006, skrypnik1994}).
\end{itemize}
\end{thm}
\begin{thm}
\label{thm:normalization}
\textbf{\emph{(Normalization property)}}
Let $\mu\in\mathcal{M}_b^p(\Omega)$ and let $U$ be any bounded 
and open subset of $\Phi^{1,p}_0(\Omega)$ containing the 
entropy solution $u$ of 
\[
\begin{cases}
- \mathrm{div}[a(x,\nabla u)] = \mu 
&\qquad\text{in $\Omega$}\,,\\
u=0
&\qquad\text{on $\partial\Omega$}\,.
\end{cases}
\]
Then
\[
\mathrm{deg}(- \mathrm{div}[a(x,\nabla u)],U,\mu) = 1\,.
\]
\end{thm}
\begin{thm}
\label{thm:existence}
\textbf{\emph{(Existence criterion)}}
Let $\mu\in\mathcal{M}_b^p(\Omega)$ and let $U$ be a bounded 
and open subset of $\Phi^{1,p}_0(\Omega)$ such 
that~\eqref{eq:bmu}
has no entropy solution $u\in\overline{U}$.
\par
Then
\[
\mathrm{deg}(- \mathrm{div}[a(x,\nabla u)] 
+ b(x,u,\nabla u),U,\mu) = 0\,.
\]
\end{thm}
\begin{thm}
\label{thm:additivity}
\textbf{\emph{(Additivity property)}}
Let $\mu\in\mathcal{M}_b^p(\Omega)$ and let $U$ be a bounded 
and open subset of $\Phi^{1,p}_0(\Omega)$ such 
that~\eqref{eq:bmu}
has no entropy solution $u\in\partial U$.
Assume that $U=U_1\cup U_2$, where $U_1, U_2$ are two disjoint 
open subsets of $\Phi^{1,p}_0(\Omega)$.
\par
Then
\begin{multline*}
\mathrm{deg}(- \mathrm{div}[a(x,\nabla u)] 
+ b(x,u,\nabla u),U,\mu) \\
= \mathrm{deg}(- \mathrm{div}[a(x,\nabla u)] 
+ b(x,u,\nabla u),U_1,\mu) \\
+ \mathrm{deg}(- \mathrm{div}[a(x,\nabla u)] 
+ b(x,u,\nabla u),U_2,\mu)\,.
\end{multline*}
\end{thm}
\begin{thm}
\label{thm:excision}
\textbf{\emph{(Excision property)}}
Let $\mu\in\mathcal{M}_b^p(\Omega)$ and let $V\subseteq U$ be 
two bounded and open subsets of $\Phi^{1,p}_0(\Omega)$ such 
that~\eqref{eq:bmu} has no entropy solution 
$u\in\overline{U}\setminus V$.
\par
Then
\begin{multline*}
 \mathrm{deg}(- \mathrm{div}[a(x,\nabla u)] 
+ b(x,u,\nabla u),U,\mu) \\
= \mathrm{deg}(- \mathrm{div}[a(x,\nabla u)] 
+ b(x,u,\nabla u),V,\mu)\,.
\end{multline*}
\end{thm}
\begin{thm}
\label{thm:homotopy}
\textbf{\emph{(Homotopy invariance property)}}
Let
\[
a:\Omega\times(\R^N\times[0,1])\rightarrow \R^N\,,\qquad
b:\Omega\times(\R\times\R^N\times[0,1])\rightarrow \R
\]
be two Carath\'eodory functions satisfying $(u_1)$ -- $(u_3)$
with respect to $T=[0,1]$ and let
$\mu_0, \mu_1\in\mathcal{M}_b^p(\Omega)$.
\par
Then the following facts hold:
\begin{itemize}
\item[$(a)$]
for every bounded and closed subset $C$ of $\Phi^{1,p}_0(\Omega)$,
the set of $t$'s in $[0,1]$ such that
\begin{equation}
\label{eq:bmut}
\begin{cases}
- \mathrm{div}[a_t(x,\nabla u)] + b_t(x,u,\nabla u)=
(1-t)\mu_0 + t\mu_1
&\quad\text{in $\Omega$}\,,\\
u=0
&\quad\text{on $\partial\Omega$}\,,
\end{cases}
\end{equation}
admits an entropy solution $u\in C$ is closed in $[0,1]$;
\item[$(b)$]
for every bounded and open subset $U$ of $\Phi^{1,p}_0(\Omega)$,
if~\eqref{eq:bmut} has no entropy solution 
with $t\in[0,1]$ and $u\in\partial U$, then
\[
\mathrm{deg}(- \mathrm{div}[a_t(x,\nabla u)]
+ b_t(x,u,\nabla u),U,(1-t)\mu_0+t\mu_1)
\]
is independent of $t\in[0,1]$.
\end{itemize}
\end{thm}
\begin{thm}
\label{thm:odd}
Let $\mu=0$ and let $U$ be a bounded 
and open subset of $\Phi^{1,p}_0(\Omega)$ such 
that~\eqref{eq:bmu}
has no entropy solution $u\in\partial U$.
Assume that $U$ is symmetric with $0\in U$ and that
\[
a(x,-\xi) = - a(x,\xi)\,,\,\,
b(x,-s,-\xi) = - b(x,s,\xi)\,,
\]
for a.e. $x\in \Omega$ and every $s\in\R$ and $\xi\in\R^N$.
\par
Then 
\[
\mathrm{deg}(- \mathrm{div}[a(x,\nabla u)] 
+ b(x,u,\nabla u),U,0)
\]
is an odd integer.
\end{thm}
%


\section{Proof of Theorem~\ref{thm:main}}
\label{sect:main}
\begin{lem}
\label{lem:apriori}
Let
\[
a:\Omega\times\R^N\rightarrow\R^N\,,\qquad
b:\Omega\times(\R\times\R^N)\rightarrow\R
\]
be two Carath\'eodory functions satisfying $(i)$--$(iii)$.
\par
Then the following facts hold:
\begin{itemize}
\item[$(a)$]
if we set
\begin{alignat*}{3}
&a_t(x,\xi) &&= 
\begin{cases}
t\,a\left(x,t^{-\frac{1}{p-1}}\,\xi\right)
&\quad\text{if $0<t\leq 1$}\,,\\
\noalign{\medskip}
a_\infty(x,\xi)
&\quad\text{if $t=0$}\,,
\end{cases}
\\
\noalign{\medskip}
&b_t(x,s,\xi) &&= 
\begin{cases}
t\,b\left(x,t^{-\frac{1}{p-1}}\,s,t^{-\frac{1}{p-1}}\,\xi\right)
&\quad\text{if $0<t\leq 1$}\,,\\
\noalign{\medskip}
b_\infty(x,s,\xi)
&\quad\text{if $t=0$}\,,
\end{cases}
\end{alignat*}
then $a_t$, $b_t$ are two Carath\'eodory functions 
satisfying~$(i)$, $(ii)$ uniformly for $t\in[0,1]$, namely
\begin{alignat}{3}
\label{eq:atmain}
&a_t(x,\xi)\cdot\xi \geq \nu |\xi|^p - \alpha_0(x)\,,\\
&|a_t(x,\xi)| \leq \alpha_1(x) + \beta|\xi|^{p-1} \,,\\
\label{eq:btmain}
&|b_t(x,s,\xi)|\leq \alpha_2(x) 
+ \beta|s|^{p-1} +\beta|\xi|^{p-1}\,,\\
&[a_t(x,\xi)-a_t(x,\eta)]\cdot(\xi-\eta) >0\,,
\end{alignat}
for a.e. $x\in\Omega$ and every $s\in\R$, $\xi, \eta\in\R^N$, 
$t\in[0,1]$ with $\xi\neq\eta$;
in particular, $a_t$, $b_t$ satisfy~$(u_1)$--$(u_3)$ with respect 
to $T=[0,1]$;
\item[$(b)$]
if the problem
\[
\begin{cases}
u\in W^{1,p}_0(\Omega)\setminus\{0\}\,,\\
\noalign{\medskip}
- \mathrm{div}[a_{\infty}(x,\nabla u)] 
+ b_{\infty}(x,u,\nabla u)=0
&\qquad\text{in $W^{-1,p'}(\Omega)$}
\end{cases}
\]
has no solution then, for every 
$\mu\in\mathcal{M}_b^p(\Omega)$,
there exists $R>0$ such that the problem
\begin{equation}
\label{eq:bmumain}
\begin{cases}
- \mathrm{div}[a_t(x,\nabla u)] + b_t(x,u,\nabla u)=
t\mu
&\quad\text{in $\Omega$}\,,\\
u=0
&\quad\text{on $\partial\Omega$}
\end{cases}
\end{equation}
has no entropy solution with $0\leq t\leq 1$ and
$u\in \Phi^{1,p}_0(\Omega)$ with
\[
\int_\Omega |\nabla[\varphi_p(u)]|^p\,dx \geq R\,.
\]
\end{itemize}
\end{lem}
\begin{proof}
Assertion~$(a)$ is easy to prove.
To prove assertion~$(b)$, assume for a contradiction that
$(u_k,t_k)$ is a sequence of solutions
of~\eqref{eq:bmumain} with $0\leq t_k\leq 1$, 
$u_k\in\Phi^{1,p}_0(\Omega)$ and 
\[
\int_\Omega |\nabla[\varphi_p(u_k)]|^p\,dx \geq k\,.
\]
By~\eqref{eq:atmain} and~\eqref{eq:estimate}, it follows
\[
\lim_k \int_\Omega |b_{t_k}(x,u_k,\nabla u_k)|\,dx = +\infty\,,
\]
whence
\[
\lim_k \int_\Omega \left(|u_k|^{p-1} 
+ |\nabla u_k|^{p-1}\right)\,dx = +\infty
\]
by~\eqref{eq:btmain}.
If we set $u_k=\tau_k^{\frac{1}{p-1}} v_k$, $\sigma_k=t_k/\tau_k$ 
with 
\begin{gather}
\tau_k = \int_\Omega \left(|u_k|^{p-1} 
+ |\nabla u_k|^{p-1}\right)\,dx\,,\\
\label{eq:vk}
\int_\Omega \left(|v_k|^{p-1} 
+ |\nabla v_k|^{p-1}\right)\,dx = 1\,,
\end{gather}
it follows that $(v_k,\sigma_k)$ satisfies
\[
\begin{cases}
- \mathrm{div}[a_{\sigma_k}(x,\nabla v_k)] 
+ b_{\sigma_k}(x,v_k,\nabla v_k)=
{\sigma_k}\mu
&\quad\text{in $\Omega$}\,,\\
v_k=0
&\quad\text{on $\partial\Omega$}
\end{cases}
\]
with $\sigma_k\to 0$.
By~\eqref{eq:vk} and Proposition~\ref{prop:Phi} we infer that
\[
\inf_k 
\int_\Omega |\nabla[\varphi_p(v_k)]|^p\,dx >0\,,
\]
while~\eqref{eq:vk}, \eqref{eq:btmain}, \eqref{eq:atmain} and
\eqref{eq:estimate} imply that
\[
\sup_k 
\int_\Omega |\nabla[\varphi_p(v_k)]|^p\,dx < +\infty\,.
\]
If we set
\[
C = \left\{u\in\Phi^{1,p}_0(\Omega):\,\,
\frac{1}{M} \leq \int_\Omega |\nabla[\varphi_p(u)]|^p\,dx
\leq M\right\}\,,
\]
where $M>0$ satisfies
\[
\frac{1}{M} \leq \int_\Omega |\nabla[\varphi_p(v_k)]|^p\,dx
\leq M\qquad\forall k\in\N\,,
\]
we have that $C$ is a bounded and closed subset of 
$\Phi^{1,p}_0(\Omega)$.
By~$(a)$ of Theorem~\ref{thm:homotopy}, there exists an entropy 
solution $v\in C\subseteq \Phi^{1,p}_0(\Omega)\setminus\{0\}$ of
\[
\begin{cases}
- \mathrm{div}[a_0(x,\nabla v)] + b_0(x,v,\nabla v)= 0
&\quad\text{in $\Omega$}\,,\\
v=0
&\quad\text{on $\partial\Omega$}\,.
\end{cases}
\]
Because of~\eqref{eq:binftyest}, we have that $b_0$
satisfies~$(a_3)$ with $\alpha_2\in L^\infty(\Omega)$.
From Theorem~\ref{thm:consistency} we infer that $v$ satisfies
\[
\begin{cases}
v\in W^{1,p}_0(\Omega)\setminus\{0\}\,,\\
\noalign{\medskip}
- \mathrm{div}[a_{\infty}(x,\nabla v)] 
+ b_{\infty}(x,v,\nabla v)=0
&\qquad\text{in $W^{-1,p'}(\Omega)$}\,,
\end{cases}
\]
and a contradiction follows.
\end{proof}
\noindent
\emph{Proof of Theorem~\ref{thm:main}.}
\par\noindent
Let $a_t$, $b_t$ be as in Lemma~\ref{lem:apriori}.
Assume that assertion~$(a)$ of Theorem~\ref{thm:main} is false 
and take $\mu\in\mathcal{M}_b^p(\Omega)$.
Then there exists $R>0$ such that~\eqref{eq:bmumain}
has no entropy solution with $0\leq t\leq 1$ and
$u\in\Phi^{1,p}_0(\Omega)$ with
\[
\int_\Omega |\nabla[\varphi_p(u)]|^p\,dx \geq R\,.
\]
If we set
\[
U = \left\{u\in\Phi^{1,p}_0(\Omega):\,\,
\int_\Omega |\nabla[\varphi_p(u)]|^p\,dx
< R\right\}\,,
\]
from~$(b)$ of Theorem~\ref{thm:homotopy} we infer that
\begin{multline*}
\mathrm{deg}(- \mathrm{div}[a(x,\nabla u)]
+ b(x,u,\nabla u),U,\mu) \\
= \mathrm{deg}(- \mathrm{div}[a_\infty(x,\nabla u)]
+ b_\infty(x,u,\nabla u),U,0)\,.
\end{multline*}
On the other hand, 
\[
\mathrm{deg}(- \mathrm{div}[a_\infty(x,\nabla u)]
+ b_\infty(x,u,\nabla u),U,0)
\]
is an odd integer by~\eqref{eq:abodd} and 
Theorem~\ref{thm:odd}, whence
\[
\mathrm{deg}(- \mathrm{div}[a(x,\nabla u)]
+ b(x,u,\nabla u),U,\mu) \neq 0\,.
\]
By Theorem~\ref{thm:existence} there exists an
entropy solution $u\in U$ of 
\[
\begin{cases}
- \mathrm{div}[a(x,\nabla u)] + b(x,u,\nabla u)=\mu
&\quad\text{in $\Omega$}\,,\\
u=0
&\quad\text{on $\partial\Omega$}
\end{cases}
\]
and assertion~$(b)$ of Theorem~\ref{thm:main} follows.
\qed


\section{Entropy solutions of parametric problems}
\label{sect:parametric}
Throughout this section, we consider a metrizable topological 
space $T$, two Carath\'eodory functions
\[
a:\Omega\times
\left(\R^n\times T\right)\rightarrow \R^n\,,\qquad
b:\Omega\times
\left(\R\times\R^n\times T\right)\rightarrow \R
\]
satisfying $(u_1)$ -- $(u_3)$ and the entropy solutions of
\begin{equation}
\label{eq:bmuT}
\begin{cases}
- \mathrm{div}[a_t(x,\nabla u)] + b_t(x,u,\nabla u)=\mu 
&\qquad\text{in $\Omega$}\,,\\
u=0
&\qquad\text{on $\partial\Omega$}\,,
\end{cases}
\end{equation}
with $\mu\in\mathcal{M}_b^p(\Omega)$.
\begin{thm}
\label{thm:proper}
Let $(t_n)$ be a sequence in $T$, $(\mu_n)$ a sequence in 
$\mathcal{M}_b^p(\Omega)$ and $(u_n)$ a sequence of entropy 
solutions of~\eqref{eq:bmuT} associated with $t_n$ and $\mu_n$.
Let also $t\in T$ and $\mu\in\mathcal{M}_b^p(\Omega)$.
Assume that $(u_n)$ is bounded in $\Phi^{1,p}_0(\Omega)$, 
$(t_n)$ is convergent to $t$ in $T$ and there exist $(w_n^{(0)})$, 
$w^{(0)}$ in $L^1(\Omega)$ and $(w_n^{(1)})$, $w^{(1)}$ in 
$L^{p'}(\Omega;\R^N)$ such that
\begin{multline*}
\int_\Omega v\,d\mu_n 
= \int_\Omega v w_n^{(0)}\,dx 
+\int_\Omega (\nabla v)\cdot w_n^{(1)}\,dx\,,\\ 
\int_\Omega v\,d\mu 
= \int_\Omega v w^{(0)}\,dx 
+\int_\Omega (\nabla v)\cdot w^{(1)}\,dx\,,\\ 
\qquad\forall v\in W^{1,p}_0(\Omega)\cap L^\infty(\Omega)\,,
\end{multline*}
and such that $(w_n^{(0)})$ is weakly convergent to
$w^{(0)}$ in $L^1(\Omega)$, while $(w_n^{(1)})$ is strongly
convergent to $w^{(1)}$ in $L^{p'}(\Omega;\R^N)$.
\par
Then there exists a subsequence $(u_{n_j})$ converging in 
$\Phi^{1,p}_0(\Omega)$ to an entropy solution $u$
of~\eqref{eq:bmuT} associated with $t$ and $\mu$.
\end{thm}
\begin{cor}
\label{cor:closed}
Let $C$ be a bounded and closed subset of $\Phi^{1,p}_0(\Omega)$
and let $\mu\in\mathcal{M}_b^p(\Omega)$.
Then the set
\[
\left\{t\in T:\,\,
\text{\eqref{eq:bmuT} admits an entropy solution $u\in C$}\right\}
\]
is closed in $T$.
\end{cor}
\begin{proof}
It is an obvious consequence of Theorem~\ref{thm:proper}.
\end{proof}
The proof of Theorem~\ref{thm:proper} will be given at the end of 
the section.
The next lemma is an adaptation of results 
of~\cite{benilan_boccardo_gallouet_gariepy_pierre_vazquez1995,
boccardo_gallouet_orsina1996, malusa_prignet2004} 
and concerns the entropy solutions of
\begin{equation}
\label{eq:b0mu}
\begin{cases}
- \mathrm{div}[a_t(x,\nabla u)] = \mu 
&\qquad\text{in $\Omega$}\,,\\
u=0
&\qquad\text{on $\partial\Omega$}\,,
\end{cases}
\end{equation}
with $\mu\in\mathcal{M}_b^p(\Omega)$.
\begin{lem}
\label{lem:comp}
Let $(t_n)$ be a sequence in $T$, $(\mu_n)$ a sequence in 
$\mathcal{M}_b^p(\Omega)$ and let $(u_n)$ be the entropy 
solution of~\eqref{eq:b0mu} associated with $t_n$ and $\mu_n$.
Assume that $(t_n)$ is convergent in $T$ and there exist
two sequences $(w_n^{(0)})$ in $L^1(\Omega)$ and $(w_n^{(1)})$ 
in $L^{p'}(\Omega;\R^N)$ such that
\[
\int_\Omega v\,d\mu_n 
= \int_\Omega v w_n^{(0)}\,dx 
+\int_\Omega (\nabla v)\cdot w_n^{(1)}\,dx
\qquad\forall v\in W^{1,p}_0(\Omega)\cap L^\infty(\Omega)\,,
\]
and such that $(w_n^{(0)})$ bounded in $L^1(\Omega)$, 
while $(w_n^{(1)})$ is strongly convergent in $L^{p'}(\Omega;\R^N)$.
\par
Then $(u_n)$ is bounded in $\Phi^{1,p}_0(\Omega)$ and
there exist $u\in \Phi^{1,p}_0(\Omega)$ and a subsequence 
$(u_{n_j})$ such that $(u_{n_j},\nabla u_{n_j})$ is convergent 
to $(u,\nabla u)$ a.e. in $\Omega$.
\end{lem}
\begin{proof}
If we set
\[
\tilde{a}_n(x,\xi) = a_{t_n}(x,\xi) - w_n^{(1)}\,,
\]
it is easily seen that each $\tilde{a}_n$ satisfies~$(a_1)$--$(a_2)$
for some $\alpha_0$, $\alpha_1$, $\beta_1$ and $\nu$ independent
of $n$.
Since $u_n$ is an entropy solution of 
\[
\begin{cases}
- \mathrm{div}[\tilde{a}_n(x,\nabla u_n)] = w_n^{(0)} 
&\qquad\text{in $\Omega$}\,,\\
u=0
&\qquad\text{on $\partial\Omega$}\,,
\end{cases}
\]
by~\eqref{eq:estimate} $(u_n)$ is bounded in 
$\Phi^{1,p}_0(\Omega)$.
\par
Therefore there exist $u\in \Phi^{1,p}_0(\Omega)$ and a 
subsequence $(u_{n_j})$ such that $(\varphi_p(u_{n_j}))$ is
convergent to $\varphi_p(u)$ weakly in $W^{1,p}_0(\Omega)$ 
and $(u_{n_j})$ is convergent to $u$ a.e. in $\Omega$.
It follows that $(T_k(u_{n_j}))$ is weakly convergent
to $T_k(u)$ in $W^{1,p}_0(\Omega)$.
\par
Now let $\vartheta:\R\rightarrow[0,1]$ be a 
$C^\infty$-function such that 
$\vartheta(s)=1$ for $|s|\leq 1$ and
$\vartheta(s)=0$ for $|s|\geq 2$.
By Theorem~\ref{thm:bb}, we have
\begin{multline*}
\int_{\Omega} a_{t_n}(x,\nabla u_n) \cdot 
\nabla \left[\vartheta\left(\frac{u_n}{k}\right)v\right]\,dx\\
=
\int_{\Omega} 
\left[\vartheta\left(\frac{u_n}{k}\right)v\right] w_n^{(0)}\,dx
+ \int_{\Omega} 
\nabla\left[\vartheta\left(\frac{u_n}{k}\right)v\right]
\cdot w_n^{(1)}\,dx
\qquad\forall v\in C^\infty_c(\Omega)\,,
\end{multline*}
namely
\[
\int_{\Omega} \hat{a}_n(x,\nabla T_{2k}(u_n)) \cdot 
\nabla v\,dx
=
\int_{\Omega} v \hat{w}_n^{(0)}\,dx
+ \int_{\Omega} \nabla v\cdot \hat{w}_n^{(1)}\,dx
\qquad\forall v\in C^\infty_c(\Omega)\,,
\]
where
\begin{alignat*}{3}
&\hat{a}_n(x,\xi) &&= 
\vartheta\left(\frac{u_n(x)}{k}\right) a_{t_n}(x,\xi)\,,\\
&\hat{w}_n^{(0)} &&=
\vartheta\left(\frac{u_n}{k}\right) w_n^{(0)} -
\vartheta'\left(\frac{u_n}{k}\right)\,
\frac{[a_{t_n}(x,\nabla T_{2k}(u_n))-w_n^{(1)}] 
\cdot\nabla T_{2k}(u_n)}{k}\,,\\
&\hat{w}_n^{(1)} &&=
\vartheta\left(\frac{u_n}{k}\right) w_n^{(1)} \,.
\end{alignat*}
If we also set
\begin{gather*}
\hat{a}(x,\xi) = 
\vartheta\left(\frac{u(x)}{k}\right) a_{t}(x,\xi)\,,\\
E_k =
\left\{x\in\Omega:\,\,|u(x)| < k\right\}\,,
\end{gather*}
we can apply~\cite[Theorem~1]{dalmaso_murat1998}.
Therefore, for any $k>0$, up to a further subsequence  
$(\nabla u_{n_j})$ is convergent to 
$\nabla u$ a.e. in $E_k$.
A standard diagonal argument shows that, up to another
subsequence,  $(\nabla u_{n_j})$ is convergent to 
$\nabla u$ a.e. in $\Omega$.
\end{proof}
\begin{lem}
\label{lem:cont}
Let $(t_n)$, $t$ be in $T$ and $(\mu_n)$, $\mu$ in 
$\mathcal{M}_b^p(\Omega)$.
Let also $(u_n)$, $u$ be the entropy solutions of~\eqref{eq:b0mu} 
associated with $t_n$, $\mu_n$ and $t$, $\mu$, respectively.
Assume that $(t_n)$ is convergent to $t$ in $T$ and there exist 
$(w_n^{(0)})$, $w^{(0)}$ in $L^1(\Omega)$ and 
$(w_n^{(1)})$, $w^{(1)}$ in $L^{p'}(\Omega;\R^N)$ such that
\begin{multline*}
\int_\Omega v\,d\mu_n 
= \int_\Omega v w_n^{(0)}\,dx 
+\int_\Omega (\nabla v)\cdot w_n^{(1)}\,dx\,,\\ 
\int_\Omega v\,d\mu 
= \int_\Omega v w^{(0)}\,dx 
+\int_\Omega (\nabla v)\cdot w^{(1)}\,dx\,,\\ 
\qquad\forall v\in W^{1,p}_0(\Omega)\cap L^\infty(\Omega)\,,
\end{multline*}
and such that $(w_n^{(0)})$ is weakly convergent to
$w^{(0)}$ in $L^1(\Omega)$, while $(w_n^{(1)})$ is strongly
convergent to $w^{(1)}$ in $L^{p'}(\Omega;\R^N)$.
\par
Then $(u_n)$ is convergent to $u$ in $\Phi^{1,p}_0(\Omega)$.
\end{lem}
\begin{proof}
First of all, by Lemma~\ref{lem:comp} $(u_n)$ is
bounded in $\Phi^{1,p}_0(\Omega)$ and there exists 
$\hat{u}\in\Phi^{1,p}_0(\Omega)$ such that, 
up to a subsequence, $(u_n,\nabla u_n)$ is convergent to
$(\hat{u},\nabla\hat{u})$ a.e. in $\Omega$.
Since $(\varphi_p(u_n))$ is weakly convergent to
$\varphi_p(\hat{u})$ in $W^{1,p}_0(\Omega)$, it follows that
$(\nabla[T_k(u_n)])$ is convergent to $\nabla[T_k(\hat{u})]$ 
weakly in $L^p(\Omega;\R^N)$ and a.e. in $\Omega$, for any $k>0$.
For any $k>0$ and $v\in C^\infty_c(\Omega)$,  it follows that
\begin{multline*}
\lim_n \int_{\Omega} T_{k}(u_n-v) \,d\mu_n =
\lim_n \int_{\Omega} \left(T_{k}(u_n-v) w_n^{(0)}
+ \nabla [T_{k}(u_n-v)]\cdot w_n^{(1)}\right)\,dx \\
=
\int_{\Omega} \left(T_{k}(\hat{u}-v) w^{(0)}
+ \nabla [T_{k}(\hat{u}-v)]\cdot w^{(1)}\right)\,dx 
= \int_{\Omega} T_{k}(\hat{u}-v) \,d\mu\,.
\end{multline*}
On the other hand, if $v\in C^\infty_c(\Omega)$ 
and $h=k+\|v\|_\infty$, we have
\[
\begin{split}
\int_{\Omega} T_{k}(u_n-v)\,d\mu_n
&\geq
\int_{\Omega} a_{t_n}(x,\nabla u_n) \cdot 
\nabla \left[T_{k}(u_n-v)\right]\,dx\\
&=
\int_{\left\{|u_n-v|<k\right\}} 
a_{t_n}(x,\nabla T_h(u_n)) \cdot \nabla T_h(u_n)\,dx \\
&\qquad\qquad\qquad
- \int_{\left\{|u_n-v|<k\right\}} 
a(x,\nabla T_h(u_n)) \cdot \nabla v\,dx \,.
\end{split}
\]
Since 
$a_{t_n}(x,\nabla T_h(u_n)) \cdot \nabla T_h(u_n)
\geq -\alpha_0$, 
we can pass to the lower limit as $n\to\infty$ and
apply Fatou's lemma, obtaining
\[
\int_{\Omega} T_{k}(\hat{u}-v) \,d\mu
\geq
\int_{\Omega} a_{t}(x,\nabla \hat{u}) \cdot 
\nabla \left[T_{k}(\hat{u}-v)\right]\,dx
\qquad\forall v\in C^\infty_c(\Omega)\,.
\]
By the uniqueness of the entropy solution we infer
that $\hat{u}=u$.
\par
According to Theorem~\ref{thm:regentr}, we also have
\begin{alignat*}{3}
&\int_{\Omega} \psi'(u_n)\,a_{t_n}(x,\nabla u_n) \cdot 
\nabla u_n\,dx &&=
\int_{\Omega} \psi(u_n)\,d\mu_n\,,\\
&\int_{\Omega} \psi'(u)\,a_{t}(x,\nabla u) \cdot 
\nabla u\,dx &&=
\int_{\Omega} \psi(u) \,d\mu\,.
\end{alignat*}
Since
\[
\psi'(u_n)\,a_{t_n}(x,\nabla u_n) \cdot 
\nabla u_n - \nu |\nabla [\varphi_p(u_n)]|^p 
\geq - \alpha_0\,,
\]
again from Fatou's lemma we infer that
\[
\begin{split}
\int_{\Omega} \psi(u) \,d\mu
- \nu \int_\Omega |\nabla[\varphi_p(u)]|^p\,dx 
&=
\int_{\Omega} \psi'(u)\,a_{t}(x,\nabla u) \cdot 
\nabla u\,dx \\
&\qquad\qquad\qquad
- \nu \int_\Omega |\nabla[\varphi_p(u)]|^p\,dx  \\
&\leq
\limsup_n \int_\Omega \psi'(u_n)\,a_{t_n}(x,\nabla u_n) 
\cdot \nabla u_n\,dx \\
&\qquad\qquad\qquad
- \nu \limsup_n \int_\Omega |\nabla[\varphi_p(u_n)]|^p\,dx \\
&=
\limsup_n \int_{\Omega} \psi(u_n) \,d\mu_n \\
&\qquad\qquad\qquad
- \nu \limsup_n \int_\Omega |\nabla[\varphi_p(u_n)]|^p\,dx \\
&=
\int_{\Omega} \psi(u) \,d\mu \\
&\qquad\qquad\qquad
- \nu \limsup_n \int_\Omega |\nabla[\varphi_p(u_n)]|^p\,dx \,.
\end{split}
\]
It follows
\[
\limsup_n \int_\Omega |\nabla[\varphi_p(u_n)]|^p\,dx \leq
\int_\Omega |\nabla[\varphi_p(u)]|^p\,dx\,,
\]
whence the strong convergence of $(u_n)$ to
$u$ in $\Phi^{1,p}_0(\Omega)$.
\end{proof}
\noindent
\emph{Proof of Theorem~\ref{thm:proper}.}
\par\noindent
Since $(u_n)$ is bounded in $\Phi^{1,p}_0(\Omega)$,
from $(b)$ of Proposition~\ref{prop:Phi} and $(u_3)$ it follows 
that $(b_{t_n}(x,u_n,\nabla u_n))$ is bounded in $L^1(\Omega)$.
By Lemma~\ref{lem:comp} we infer that there exists
$u\in \Phi^{1,p}_0(\Omega)$ such that, up to a subsequence,
$(u_n,\nabla u_n)$ is convergent to $(u,\nabla u)$ 
a.e. in $\Omega$.
\par
From Proposition~\ref{prop:Phi} and $(u_3)$ we deduce 
that $(b_{t_n}(x,u_n,\nabla u_n))$ is (strongly) convergent to 
$b_t(x,u,\nabla u)$ in $L^1(\Omega)$.
By Lemma~\ref{lem:cont} we conclude that
$(u_n)$ is convergent to $u$ in $\Phi^{1,p}_0(\Omega)$ and that
$u$ is an entropy solution of~\eqref{eq:bmuT}.
\qed
%


\section{Degree theory in reflexive Banach spaces}
\label{sect:degreeb}
Let $X$ be a reflexive real Banach space.
\begin{defn}
\label{defn:S+}
A map $F:D\rightarrow X'$, with $D\subseteq X$, is said
to be \emph{of class~$(S)_+$} if, for every sequence $(u_n)$
in $D$ weakly converging to some $u$ in $X$ with
\[
\limsup_n \, \langle F(u_n),u_n-u\rangle \leq 0\,,
\]
it holds $\|u_n-u\|\to 0$.
\par
More generally, if $T$ is a metrizable topological space,
a map $H:D\rightarrow X'$, with $D\subseteq X\times T$, is said
to be \emph{of class~$(S)_+$} if, for every sequence
$(u_n,t_n)$ in $D$ with
$(u_n)$ weakly converging to $u$ in $X$,
$(t_n)$ converging to $t$ in $T$ and
\[
\limsup_n \, \langle H_{t_n}(u_n),u_n-u\rangle \leq 0\,,
\]
it holds $\|u_n-u\|\to 0$ (we write $H_t(u)$
instead of $H(u,t)$).
\end{defn}
Assume now that $U$ is a bounded and open subset 
of $X$, $F:\cl{U}\rightarrow X'$ a 
continuous map of class~$(S)_+$ and $w\in X'$.
\begin{rem}
It is easily seen that the set
\[
\left\{u\in\cl{U}:\,\,F(u)=w\right\}
\]
is compact (possibly empty).
\end{rem}
According
to~\cite{browder1983, oregan_cho_chen2006, skrypnik1994},
if $w\not\in F(\partial U)$, one can define
the topological degree
\[
\mathrm{deg}_{(S)_+}(F,U,w)\in\Z\,.
\]
Let us recall some basic properties.
\begin{prop}
If $w\not\in F(\partial U)$, then
\[
\mathrm{deg}_{(S)_+}(F,U,w) = \mathrm{deg}_{(S)_+}(F-w,U,0)\,.
\]
\end{prop}
\begin{thm}
\label{thm:normalizationw}
If $w\not\in F(\partial U)$, $u_0\in U$ and
\[
\langle F(u),u-u_0\rangle \geq \langle w,u-u_0\rangle
\qquad\text{for any $u\in \partial{U}$}\,,
\]
then $\mathrm{deg}_{(S)_+}(F,U,w) = 1$.
\end{thm}
\begin{thm}
If $w\not\in F(\cl{U})$, then
$\mathrm{deg}_{(S)_+}(F,U,w) = 0$.
\end{thm}
\begin{thm}
If $w\not\in F(\partial U)$
and $U=U_0\cup U_1$, where $U_0, U_1$ are two disjoint open
subsets of $X$, then
\[
\mathrm{deg}_{(S)_+}(F,U,w) =
\mathrm{deg}_{(S)_+}(F,U_0,w) + 
\mathrm{deg}_{(S)_+}(F,U_1,w) \,.
\]
\end{thm}
\begin{thm}
\label{thm:excisionw}
Let $V$ be another open subset of $X$ with
$V\subseteq U$ and assume that 
$w\not\in F(\overline{U}\setminus V)$.
\par
Then $\mathrm{deg}_{(S)_+}(F,U,w) = 
\mathrm{deg}_{(S)_+}(F,V,w)$.
\end{thm}
\begin{thm}
\label{thm:homotopyw}
Let $H:\overline{U}\times[0,1]\rightarrow X'$
be a continuous map of class~$(S)_+$.
Then the following facts hold:
\begin{itemize}
\item[$(a)$]
the set
\[
\left\{(u,t)\in\cl{U}\times[0,1]:\,\,
H_t(u)=w\right\}
\]
is compact (possibly empty);
\item[$(b)$]
if $w\not\in H(\partial{U}\times[0,1])$,
then $\mathrm{deg}_{(S)_+}(H_{t},U,w)$ is independent of
$t\in[0,1]$.
\end{itemize}
\end{thm}
\begin{thm}
\label{thm:oddw}
Assume that $U$ is symmetric with $0\in U$ and that $F$
is odd with $0\not\in F(\partial U)$.
\par
Then 
$\mathrm{deg}_{(S)_+}(F,U,0)$
is an odd integer.
\end{thm}
Let us also introduce a variant, more in the line
of the degree for ``compactly rooted maps'' 
of~\cite{dugundji_granas2003}.
\par
Assume that $U$ is a (possibly unbounded) open subset of $X$, 
$F:U\rightarrow X'$ is continuous and locally of class~$(S)_+$ 
and $w\in X'$.
\begin{prop}
\label{prop:c}
If 
\begin{equation}
\label{eq:compactsol}
\text{$\left\{u\in U:\,\,F(u)=w\right\}$
is compact (possibly empty)}\,,
\end{equation}
then the following facts hold:
\begin{itemize}
\item[$(a)$]
there exists a bounded and open subset $V$ of $X$ such that
$\cl{V}\subseteq U$, $F$ is of class~$(S)_+$ on $\cl{V}$
and $w\not\in F(U\setminus V)$;
\item[$(b)$]
if $V_0$ and $V_1$ are as in~$(a)$, it holds
\[
\mathrm{deg}_{(S)_+}(F,V_0,w) = \mathrm{deg}_{(S)_+}(F,V_1,w)\,.
\]
\end{itemize}
\end{prop}
\begin{proof}
Assertion~$(a)$ is easy to prove. 
If $V_0$ and $V_1$ are as in~$(a)$, from 
Theorem~\ref{thm:excisionw} we infer that
\[
\mathrm{deg}_{(S)_+}(F,V_0,w) =
\mathrm{deg}_{(S)_+}(F,V_0\cap V_1,w) =
\mathrm{deg}_{(S)_+}(F,V_1,w)
\]
and assertion~$(b)$ also follows.
\end{proof}
Therefore, if~\eqref{eq:compactsol} holds, one can define
$\widetilde{\mathrm{deg}}_{(S)_+}(F,U,w)$ as
\[
\widetilde{\mathrm{deg}}_{(S)_+}(F,U,w) =
\mathrm{deg}_{(S)_+}(F,V,w)\,,
\]
where $V$ is any bounded and open subset of $X$ as in~$(a)$
of Proposition~\ref{prop:c}.
\par
The next results are easy consequences of the properties
of the degree in the previous setting. 
\begin{prop}
Assume that $U$ is bounded, $F:\cl{U}\rightarrow X'$ is 
continuous and of class~$(S)_+$ and 
$w\in X'\setminus F(\partial U)$.
\par
Then~\eqref{eq:compactsol} holds and we have
\[
\widetilde{\mathrm{deg}}_{(S)_+}(F,U,w) =
\mathrm{deg}_{(S)_+}(F,U,w)\,,
\]
\end{prop}
\begin{prop}
If~\eqref{eq:compactsol} holds, then
\[
\widetilde{\mathrm{deg}}_{(S)_+}(F,U,w) = 
\widetilde{\mathrm{deg}}_{(S)_+}(F-w,U,0)\,.
\]
\end{prop}
\begin{thm}
\label{thm:normalizationc}
If~\eqref{eq:compactsol} holds, $u_0\in U$ and
\[
\langle F(u),u-u_0\rangle \geq \langle w,u-u_0\rangle
\qquad\text{for any $u\in U$}\,,
\]
then $\widetilde{\mathrm{deg}}_{(S)_+}(F,U,w) = 1$.
\end{thm}
\begin{thm}
\label{thm:existencec}
If 
\[
\left\{u\in U:\,\,F(u)=w\right\}
\]
is empty, then $\widetilde{\mathrm{deg}}_{(S)_+}(F,U,w) = 0$.
\end{thm}
\begin{thm}
\label{thm:additivityc}
If~\eqref{eq:compactsol} holds
and $U=U_0\cup U_1$, where $U_0, U_1$ are two disjoint open
subsets of $X$, then
\[
\widetilde{\mathrm{deg}}_{(S)_+}(F,U,w) =
\widetilde{\mathrm{deg}}_{(S)_+}(F,U_0,w) + 
\widetilde{\mathrm{deg}}_{(S)_+}(F,U_1,w) \,.
\]
\end{thm}
\begin{thm}
\label{thm:excisionc}
If~\eqref{eq:compactsol} holds and $V$ is another open 
subset of $X$ with $V\subseteq U$ and
$w\not\in F(U\setminus V)$,
then $\widetilde{\mathrm{deg}}_{(S)_+}(F,U,w) = 
\widetilde{\mathrm{deg}}_{(S)_+}(F,V,w)$.
\end{thm}
\begin{thm}
\label{thm:homotopyc}
Assume that $H:U\times[0,1]\rightarrow X'$
is continuous and locally of class~$(S)_+$
and that
\[
\left\{(u,t)\in U\times[0,1]:\,\,
H_t(u)=w\right\}
\]
is compact (possibly empty).
\par
Then $\widetilde{\mathrm{deg}}_{(S)_+}(H_{t},U,w)$ is 
independent of $t\in[0,1]$.
\end{thm}
\begin{thm}
\label{thm:oddc}
If~\eqref{eq:compactsol} holds with $w=0$, $U$ is 
symmetric with $0\in U$ and $F$ is odd, then
$\widetilde{\mathrm{deg}}_{(S)_+}(F,U,0)$
is an odd integer.
\end{thm}
%


\section{The construction of the degree}
\label{sect:degreeeq}
Now consider again two Carath\'eodory functions
\[
a:\Omega\times\R^N\rightarrow\R^N\,,\qquad
b:\Omega\times(\R\times\R^N)\rightarrow\R
\]
satisfying $(a_1)$ -- $(a_3)$.
\par
We first treat a particular case of~\eqref{eq:bmu}, namely 
\begin{equation}
\label{eq:b0}
\begin{cases}
- \mathrm{div}[a(x,\nabla u)] + b(x,u,\nabla u)=0
&\qquad\text{in $\Omega$}\,,\\
u=0
&\qquad\text{on $\partial\Omega$}\,.
\end{cases}
\end{equation}
If we set
\begin{alignat*}{3}
& a_{\tau}(x,\xi) &&= a(x,\xi)
&&\qquad\text{if $0\leq \tau\leq 1$}\,,\\
& b_{\tau}(x,s,\xi) &&= 
T_{1/\tau}\left(b(x,s,\xi)\right)
&&\qquad\text{if $0<\tau\leq 1$}\,,\\
& b_{\tau}(x,s,\xi) &&=
b(x,s,\xi)
&&\qquad\text{if $\tau=0$}\,,
\end{alignat*}
it is easily seen that $a_\tau, b_\tau$ satisfy
$(u_1)$--$(u_3)$ with respect to $T=[0,1]$.
Therefore we can consider the entropy solutions of
\begin{equation}
\label{eq:entropytau}
\begin{cases}
- \mathrm{div}[a(x,\nabla u)] + b_\tau(x,u,\nabla u)=0 
&\qquad\text{in $\Omega$}\,,\\
u=0
&\qquad\text{on $\partial\Omega$}\,.
\end{cases}
\end{equation}
Moreover, for every $\underline{\tau}\in]0,1[$, we have
\[
|b_\tau(x,s,\xi)| \leq \frac{1}{\underline{\tau}}
\qquad\text{whenever $\underline{\tau}\leq \tau\leq 1$}\,.
\]
Therefore, we can define a continuous map
\[
H:W^{1,p}_0(\Omega)\times ]0,1]\rightarrow W^{-1,p'}(\Omega)
\]
by
\[
H_\tau(u) = - \mathrm{div}\left[a(x,\nabla u)\right]
+ b_\tau(x,u,\nabla u)
\]
and, according
to~\cite{browder1983, skrypnik1994}, 
this map is of class~$(S)_+$
(see also~\cite[Theorem~3.5]{almi_degiovanni2013}).
Observe also that, if $U$ is an open subset of 
$\Phi^{1,p}_0(\Omega)$, then $U\cap W^{1,p}_0(\Omega)$ is an 
open subset of $W^{1,p}_0(\Omega)$.
\begin{prop}
\label{prop:defndeg}
Let $U$ be a bounded and open subset of $\Phi^{1,p}_0(\Omega)$
such that~\eqref{eq:b0} has no entropy solution  
$u\in\partial U$.
\par
Then the following facts hold:
\begin{itemize}
\item[$(a)$]
there exists $\overline{\tau}\in]0,1]$ such 
that~\eqref{eq:entropytau} has no entropy solution 
with $0\leq \tau \leq\overline{\tau}$ and $u\in\partial U$;
\item[$(b)$]
if $\overline{\tau}\in]0,1]$ is like in~$(a)$, then
for every $\underline{\tau} \in]0,\overline{\tau}[$ the set
\[
\left\{(u,\tau)\in (U\cap W^{1,p}_0(\Omega))
\times [\underline{\tau},\overline{\tau}]:\,\,
H_{\tau}(u)=0\right\}
\]
is compact in $W^{1,p}_0(\Omega)\times
[\underline{\tau},\overline{\tau}]$ and the topological degree 
\[
\widetilde{\mathrm{deg}}_{(S)_+}(H_{\tau},U\cap W^{1,p}_0(\Omega),0)
\]
is constant for $\tau \in]0,\overline{\tau}]$.
\end{itemize}
\end{prop}
\begin{proof}
Since $\partial U$ is closed and bounded in 
$\Phi^{1,p}_0(\Omega)$, assertion~$(a)$ follows from
Corollary~\ref{cor:closed}.
\par
On the other hand, for every 
$\underline{\tau} \in]0,\overline{\tau}[$,
if $H_\tau(u)=0$ with $\tau\in [\underline{\tau},\overline{\tau}]$
and $u\in W^{1,p}_0(\Omega)$,
it turns out that $u$ solves a problem of the form
\[
\begin{cases}
- \mathrm{div}[a(x,\nabla u)] = z
&\qquad\text{in $\Omega$}\,,\\
u=0
&\qquad\text{on $\partial\Omega$}\,,
\end{cases}
\]
with $z\in L^\infty(\Omega)$ and 
$\|z\|_\infty \leq 1/\underline{\tau}$.
Since $\left\{(u,\tau)\mapsto H_\tau(u)\right\}$ is
of class~$(S)_+$, it easily follows that 
\[
\left\{(u,\tau)\in W^{1,p}_0(\Omega)\times
[\underline{\tau},\overline{\tau}]:\,\,H_\tau(u)=0\right\}
\]
is compact in $W^{1,p}_0(\Omega)\times
[\underline{\tau},\overline{\tau}]$.
Moreover, $H_\tau(u)=0$ implies that $(u,\tau)$ is an entropy 
solution of~\eqref{eq:entropytau}, so that there are no 
solutions of $H_\tau(u)=0$ with $u$ on the boundary of
$U\cap W^{1,p}_0(\Omega)$ in $W^{1,p}_0(\Omega)$. 
Therefore
\[
\left\{(u,\tau)\in (U\cap W^{1,p}_0(\Omega))\times
[\underline{\tau},\overline{\tau}]:\,\,H_\tau(u)=0\right\}
\]
also is compact in $W^{1,p}_0(\Omega)\times
[\underline{\tau},\overline{\tau}]$.
By Theorem~\ref{thm:homotopyc} we conclude that
\[
\widetilde{
\mathrm{deg}}_{(S)_+}(H_{\tau},U\cap W^{1,p}_0(\Omega),0)
\]
is constant for $\tau \in[\underline{\tau},\overline{\tau}]$,
whenever $\underline{\tau}\in]0,\overline{\tau}[$.
\end{proof}
\begin{defn}
\label{defn:degree0}
Let $U$ be a bounded and open subset of $\Phi^{1,p}_0(\Omega)$
such that~\eqref{eq:b0} has no entropy solution  
$u\in\partial U$.
We set
\[
\mathrm{deg}(- \mathrm{div}[a(x,\nabla u)] 
+ b(x,u,\nabla u),U,0)
= 
\widetilde{\mathrm{deg}}_{(S)_+}(H_{\overline{\tau}},
U\cap W^{1,p}_0(\Omega),0)\,,
\]
where $\overline{\tau}\in]0,1]$ is any number as
in~$(a)$ of the previous Proposition.
\end{defn}
\begin{thm}
\label{thm:consistency0}
Suppose that 
$\alpha_2\in L^1(\Omega)\cap W^{-1,p'}(\Omega)$ 
in assumption~$(a_3)$.
\par
Then the following facts hold:
\begin{itemize}
\item[$(a)$]
we have
\[
\begin{cases}
b(x,u,\nabla u)v\in L^1(\Omega) \\
\noalign{\medskip}
b(x,u,\nabla u)\in L^1(\Omega)\cap W^{-1,p'}(\Omega) 
\end{cases}
\qquad\text{for any $u,v\in W^{1,p}_0(\Omega)$}
\]
and the map
\[
\begin{array}{ccc}
W^{1,p}_0(\Omega) & \longrightarrow & W^{-1,p'}(\Omega) \\
\noalign{\medskip}
u & \mapsto & -\mathrm{div}[a(x,\nabla u)]+b(x,u,\nabla u)
\end{array}
\]
is continuous and of class~$(S)_+$;
\item[$(b)$]
every entropy solution of~\eqref{eq:b0} belongs to 
$W^{1,p}_0(\Omega)$ and every $u\in W^{1,p}_0(\Omega)$
is an entropy solution of~\eqref{eq:b0} if and only if
\[
- \mathrm{div}[a(x,\nabla u)] + b(x,u,\nabla u)=0
\qquad\text{in $W^{-1,p'}(\Omega)$}\,;
\]
\item[$(c)$]
if $U$ is a bounded and open subset of $\Phi^{1,p}_0(\Omega)$
such that~\eqref{eq:b0} has no entropy solution  
$u\in\partial U$, then the set
\[
\left\{u\in U:\,\,
- \mathrm{div}[a(x,\nabla u)] + b(x,u,\nabla u)=0\right\}
\]
is compact in $W^{1,p}_0(\Omega)$ and we have
\begin{multline*}
\null\qquad\qquad
\mathrm{deg}(- \mathrm{div}[a(x,\nabla u)] 
+ b(x,u,\nabla u),U,0) \\
= \mathrm{deg}_{(S)_+}(- \mathrm{div}[a(x,\nabla u)] 
+ b(x,u,\nabla u),U\cap V,0)\,,
\end{multline*}
whenever $V$ is a bounded and open subset of $W^{1,p}_0(\Omega)$
such that there are no solutions of~\eqref{eq:b0}
in $U\setminus V$.
\end{itemize}
\end{thm}
\begin{proof}
First of all, it is easily seen that
$\chi_{\{\alpha_2\geq k\}}\alpha_2
\in L^1(\Omega)\cap W^{-1,p'}(\Omega)$ for any $k>0$.
We claim that
\begin{equation}
\label{eq:alpha2}
\lim_{k\to+\infty} \chi_{\{\alpha_2\geq k\}}\alpha_2 = 0
\qquad\text{strongly in $W^{-1,p'}(\Omega)$}\,.
\end{equation}
Actually, let
$v_k\in W^{1,p}_0(\Omega)\cap L^\infty(\Omega)$ be such that
$\|\nabla v_k\|_p\leq 1$ and
\[
\left|\int_{\{\alpha_2\geq k\}} \alpha_2\,v_k\,dx\right| \geq
\frac{1}{2}\,\|\chi_{\{\alpha_2\geq k\}}\alpha_2\|_{-1,p'}  \,.
\]
Up to a subsequence, $(v_k)$ is convergent to some 
$v$ weakly in $W^{1,p}_0(\Omega)$ and a.e. in~$\Omega$,
so that $(|v_k|)$ is convergent to  
$|v|$ weakly in $W^{1,p}_0(\Omega)$.
Now $(\chi_{\{\alpha_2\geq k\}}\alpha_2 v_k)$
is convergent to $0$ a.e. in $\Omega$ and is dominated by
$(\alpha_2 |v_k|)$, which is convergent to
$\alpha_2 |v|$ a.e. in $\Omega$.
According to~\cite{brezis_browder1978}, we also have
$\alpha_2 |v|\in L^1(\Omega)$ and
\[
\int_{\Omega} \alpha_2|v|\,dx
= \langle\alpha_2,|v|\rangle
= \lim_k \,\langle\alpha_2,|v_k|\rangle 
= \lim_k \int_{\Omega} \alpha_2|v_k|\,dx \,.
\]
By a variant of Lebesgue's theorem, we infer that
\[
\lim_k \,\int_{\{\alpha_2\geq k\}} \alpha_2\,v_k\,dx = 0
\]
and~\eqref{eq:alpha2} follows.
\par
Now let $b_\tau$ be as before and let
\begin{alignat*}{3}
&b_{\tau,1}(x,s,\xi) &&=
\min\left\{\max\left\{b_\tau(x,s,\xi),-\alpha_2(x)\right\},
\alpha_2(x)\right\}\,,\\
&b_{\tau,2}(x,s,\xi) &&=
b_\tau(x,s,\xi) - b_{\tau,1}(x,s,\xi)\,.
\end{alignat*}
It is easily seen that $b_{\tau,1}$, $b_{\tau,2}$ are Carath\'eodory
functions satisfying
\begin{alignat*}{3}
&b_\tau(x,s,\xi) &&= b_{\tau,1}(x,s,\xi) + b_{\tau,2}(x,s,\xi)\,,\\
&|b_{\tau,1}(x,s,\xi)| &&\leq \alpha_2(x)\,,\\
&|b_{\tau,2}(x,s,\xi)| &&\leq \beta_2|s|^r +\beta_2|\xi|^q\,.
\end{alignat*}
Moreover, if $u\in \Phi^{1,p}_0(\Omega)$ and 
$v\in W^{1,p}_0(\Omega)\cap L^\infty(\Omega)$, we have
\[
\left|\int_\Omega b_{\tau,1}(x,u,\nabla u)\,v\,dx\right| \leq
\int_\Omega \alpha_2\,|v|\,dx \\
\leq
\|\alpha_2\|_{-1,p'} \|[\nabla|v|]\|_p \\
=
\|\alpha_2\|_{-1,p'} \|\nabla v\|_p \,.
\]
Therefore, it holds 
$b_{\tau,1}(x,u,\nabla u)\in L^1(\Omega)\cap W^{-1,p'}(\Omega)$ with
\[
\|b_{\tau,1}(x,u,\nabla u)\|_{-1,p'} \leq \|\alpha_2\|_{-1,p'}\,.
\]
Moreover, again by~\cite{brezis_browder1978}, we have
$\alpha_2|v|\in L^1(\Omega)$, 
hence $b_{\tau,1}(x,u,\nabla u)v\in L^1(\Omega)$,
for any $v\in W^{1,p}_0(\Omega)$.
\par
If we set
\[
b_{\tau,1}^{(k)}(x,s,\xi) = 
\chi_{\{\alpha_2<k\}}(x)\, b_{\tau,1}(x,s,\xi)\,,
\]
then $b_{\tau,1}^{(k)}$ is a Carath\'eodory function and the 
same argument shows that
\[
\|b_{\tau,1}^{(k)}(x,u,\nabla u)-b_{\tau,1}(x,u,\nabla u)\|_{-1,p'} 
\leq \|\chi_{\{\alpha_2\geq k\}}\alpha_2\|_{-1,p'}\,,
\]
so that
\[
\lim_k \|b_{\tau,1}^{(k)}(x,u,\nabla u)
- b_{\tau,1}(x,u,\nabla u)\|_{-1,p'} = 0 
\qquad\text{uniformly for $u\in \Phi^{1,p}_0(\Omega)$}
\]
by~\eqref{eq:alpha2}.
Since each map
\[
\begin{array}{ccc}
W^{1,p}_0(\Omega) \times[0,1]& \longrightarrow 
& W^{-1,p'}(\Omega) \\
\noalign{\medskip}
(u,\tau) & \mapsto & b_{\tau,1}^{(k)}(x,u,\nabla u)
\end{array}
\]
is completely continuous, as $|b_{\tau,1}^{(k)}(x,s,\xi)|\leq k$,
it follows that 
\[
\begin{array}{ccc}
W^{1,p}_0(\Omega) \times[0,1]& \longrightarrow 
& W^{-1,p'}(\Omega) \\
\noalign{\medskip}
(u,\tau) & \mapsto & b_{\tau,1}(x,u,\nabla u)
\end{array}
\]
is completely continuous, too.
\par
On the other hand, it is standard that
\[
\begin{cases}
b_{\tau,2}(x,u,\nabla u)v\in L^1(\Omega) \\
\noalign{\medskip}
b_{\tau,2}(x,u,\nabla u)\in L^1(\Omega)\cap W^{-1,p'}(\Omega) 
\end{cases}
\qquad\text{for any $u,v\in W^{1,p}_0(\Omega)$}
\]
and that the map
\[
\begin{array}{ccc}
W^{1,p}_0(\Omega)\times[0,1] & \longrightarrow 
& W^{-1,p'}(\Omega) \\
\noalign{\medskip}
(u,\tau) & \mapsto & 
-\mathrm{div}[a(x,\nabla u)+b_{\tau,2}(x,u,\nabla u)
\end{array}
\]
is continuos and of class~$(S)_+$
(see e.g.~\cite{browder1983, skrypnik1994}).
Then
\[
\begin{array}{ccc}
W^{1,p}_0(\Omega)\times[0,1] & \longrightarrow 
& W^{-1,p'}(\Omega) \\
\noalign{\medskip}
(u,\tau) & \mapsto & 
-\mathrm{div}[a(x,\nabla u)+b_\tau(x,u,\nabla u)
\end{array}
\]
is continuos and of class~$(S)_+$, too, and assertion~$(a)$ 
follows.
\par
For every $u\in \Phi^{1,p}_0(\Omega)$ and $\tau\in[0,1]$, 
there exists $w_{u,\tau}\in L^{p'}(\Omega;\R^N)$ such that 
\[
b_{\tau,1}(x,u,\nabla u) = \mathrm{div}\,w_{u,\tau}
\]
and $\|w_{u,\tau}\|_{p'}$ is bounded by a constant independent
of $u$ and $\tau$.
Therefore, each entropy solution of~\eqref{eq:entropytau}
is also an entropy solution of 
\[
\begin{cases}
- \mathrm{div}[\tilde{a}(x,\nabla u)] + b_{\tau,2}(x,u,\nabla u)=0 
&\qquad\text{in $\Omega$}\,,\\
u=0
&\qquad\text{on $\partial\Omega$}\,,
\end{cases}
\]
with
\[
\tilde{a}(x,\xi) = a(x,\xi) - w_{u,\tau}(x)  \,,
\]
and $\tilde{a}$ also satisfies $(a_1)$ and $(a_2)$, possibly
with different $\alpha_0$, $\alpha_1$, $\beta_1$ and $\nu$,
where $\|\alpha_0\|_1$, $\|\alpha_1\|_{p'}$, $\beta_1$ and 
$1/\nu$ are bounded by a constant independent of $u$ and $\tau$.
By Theorem~\ref{thm:reglp} and a standard bootstrap argument,
we infer that $u\in W^{1,p}_0(\Omega)$.
Moreover, if $B$ is a bounded subset of $\Phi^{1,p}_0(\Omega)$,
we have that
\[
\left\{(u,\tau)\in B\times[0,1]:\,\,
\text{$(u,\tau)$ is an entropy solution 
of~\eqref{eq:entropytau}}\right\}
\]
is bounded in $W^{1,p}_0(\Omega)\times[0,1]$.
\par
As
in~\cite{benilan_boccardo_gallouet_gariepy_pierre_vazquez1995},
any $u\in W^{1,p}_0(\Omega)$
is an entropy solution of~\eqref{eq:b0} if and only if
\[
- \mathrm{div}[a(x,\nabla u)] + b(x,u,\nabla u)=0
\qquad\text{in $W^{-1,p'}(\Omega)$}
\]
and assertion~$(b)$ follows.
\par
To prove assertion~$(c)$, observe that,
since $U$ is bounded in $\Phi^{1,p}_0(\Omega)$, we have 
that the set
\[
\left\{(u,\tau)\in U\times[0,1]:\,\,
\text{$(u,\tau)$ is an entropy solution 
of~\eqref{eq:entropytau}}\right\}
\]
is bounded in $W^{1,p}_0(\Omega)\times[0,1]$.
\par
According to Proposition~\ref{prop:defndeg}, 
let now $\overline{\tau}\in]0,1]$ be such 
that~\eqref{eq:entropytau} has no entropy solution 
with $0\leq \tau \leq\overline{\tau}$ and $u\in\partial U$.
Since the map
\[
\left\{(u,\tau)\mapsto 
-\mathrm{div}[a(x,\nabla u)+b_\tau(x,u,\nabla u)\right\}
\]
is continuous and of class~$(S)_+$, the set
\[
\left\{(u,\tau)\in U\times[0,\overline{\tau}]:\,\,
\text{$(u,\tau)$ is an entropy solution 
of~\eqref{eq:entropytau}}\right\}
\]
is even compact in $W^{1,p}_0(\Omega)\times[0,\overline{\tau}]$.
Let $V$ be a bounded and open subset of $W^{1,p}_0(\Omega)$
such that there are no solutions of~\eqref{eq:b0}
in $U\setminus V$ and let $W$ be a bounded and open subset of 
$W^{1,p}_0(\Omega)$ such that 
$\overline{W}\subseteq U$ and there are 
no solutions of~\eqref{eq:entropytau} with
$0\leq t\leq \overline{\tau}$ and $u\in U\setminus W$.
Combining Definition~\ref{defn:degree0} with
Theorems~\ref{thm:excisionw} and~\ref{thm:homotopyw},
we conclude that
\begin{multline*}
\mathrm{deg}(- \mathrm{div}[a(x,\nabla u)] + b(x,u,\nabla u),U,0) \\
= \widetilde{\mathrm{deg}}_{(S)_+}(H_{\overline{\tau}},
U\cap W^{1,p}_0(\Omega),0) 
= \mathrm{deg}_{(S)_+}(H_{\overline{\tau}},W,0) 
\qquad\qquad\qquad\null
\\
= \mathrm{deg}_{(S)_+}(- \mathrm{div}[a(x,\nabla u)] 
+ b(x,u,\nabla u),W,0) 
\qquad\qquad\null
\\
\null\qquad\qquad
= \mathrm{deg}_{(S)_+}(- \mathrm{div}[a(x,\nabla u)] 
+ b(x,u,\nabla u),W\cap V,0) \\
= \mathrm{deg}_{(S)_+}(- \mathrm{div}[a(x,\nabla u)] 
+ b(x,u,\nabla u),U\cap V,0) 
\end{multline*}
and assertion~$(c)$ follows.
\end{proof}
\begin{thm}
\label{thm:existence0}
Let $U$ be a bounded and open subset of $\Phi^{1,p}_0(\Omega)$
such that the equation~\eqref{eq:b0} has no entropy 
solution $u\in\overline{U}$.
\par
Then
\[
\mathrm{deg}(- \mathrm{div}[a(x,\nabla u)] 
+ b(x,u,\nabla u),U,0) = 0\,.
\]
\end{thm}
\begin{proof}
Since $\overline{U}$ is closed and bounded in 
$\Phi^{1,p}_0(\Omega)$, by Corollary~\ref{cor:closed} there 
exists $\overline{\tau}\in]0,1]$ such that~\eqref{eq:entropytau}
has no entropy solution with $0\leq \tau \leq\overline{\tau}$
and $u\in\cl{U}$.
In particular, the equation $H_{\overline{\tau}}(u)=0$ 
has no solution $u\in U\cap W^{1,p}_0(\Omega)$.
By Theorem~\ref{thm:existencec} it follows
\[
\mathrm{deg}(- \mathrm{div}[a(x,\nabla u)] 
+ b(x,u,\nabla u),U,0) =
\widetilde{\mathrm{deg}}_{(S)_+}(H_{\overline{\tau}},
U\cap W^{1,p}_0(\Omega),0) = 0\,.
\]
\end{proof}
In a similar way, Theorem~\ref{thm:odd} and the next two results 
can be proved taking advantage
of~Theorems~\ref{thm:oddc}, \ref{thm:additivityc}, 
\ref{thm:excisionc} and Corollary~\ref{cor:closed}.
\begin{thm}
\label{thm:additivity0}
Let $U$ be a bounded and open subset of $\Phi^{1,p}_0(\Omega)$
such that~\eqref{eq:b0} has no entropy solution  
$u\in\partial U$.
Assume that $U=U_1\cup U_2$, where $U_1, U_2$ are two disjoint 
open subsets of $\Phi^{1,p}_0(\Omega)$.
\par
Then
\begin{multline*}
\mathrm{deg}(- \mathrm{div}[a(x,\nabla u)] 
+ b(x,u,\nabla u),U,0) \\
= \mathrm{deg}(- \mathrm{div}[a(x,\nabla u)] 
+ b(x,u,\nabla u),U_1,0) \\
+ \mathrm{deg}(- \mathrm{div}[a(x,\nabla u)] 
+ b(x,u,\nabla u),U_2,0)\,.
\end{multline*}
\end{thm}
\begin{thm}
\label{thm:excision0}
Let $V\subseteq U$ be two bounded and open subsets of 
$\Phi^{1,p}_0(\Omega)$ such that~\eqref{eq:b0} has no
entropy solution $u\in\overline{U}\setminus V$.
\par
Then
\begin{multline*}
 \mathrm{deg}(- \mathrm{div}[a(x,\nabla u)] 
+ b(x,u,\nabla u),U,0) \\
= \mathrm{deg}(- \mathrm{div}[a(x,\nabla u)] 
+ b(x,u,\nabla u),V,0)\,.
\end{multline*}
\end{thm}
Let us see more in detail the homotopy invariance.
\begin{thm}
\label{thm:homotopy0}
Let
\[
a:\Omega\times(\R^N\times[0,1])\rightarrow \R^N\,,\qquad
b:\Omega\times(\R\times\R^N\times[0,1])\rightarrow \R
\]
be two Carath\'eodory functions satisfying $(u_1)$--$(u_3)$.
\par
Then the following facts hold:
\begin{itemize}
\item[$(a)$]
for every bounded and closed subset $C$ of $\Phi^{1,p}_0(\Omega)$,
the set of $t$'s in $[0,1]$ such that
\begin{equation}
\label{eq:bt0h}
\begin{cases}
- \mathrm{div}[a_t(x,\nabla u)] + b_t(x,u,\nabla u)=0 
&\qquad\text{in $\Omega$}\,,\\
u=0
&\qquad\text{on $\partial\Omega$}\,,
\end{cases}
\end{equation}
admits an entropy solution $u\in C$ is closed in $[0,1]$;
\item[$(b)$]
for every bounded and open subset $U$ of $\Phi^{1,p}_0(\Omega)$,
if~\eqref{eq:bt0h} has no entropy solution 
with $t\in[0,1]$ and $u\in\partial U$, then
\[
\mathrm{deg}(- \mathrm{div}[a_t(x,\nabla u)] 
+ b_t(x,u,\nabla u),U,0)
\]
is independent of $t\in[0,1]$.
\end{itemize}
\end{thm}
\begin{proof}
Assertion~$(a)$ is a particular case of
Corollary~\ref{cor:closed}.
To prove~$(b)$, consider
\begin{alignat*}{3}
& a_{t,\tau}(x,\xi) &&= a_t(x,\xi)
&&\qquad\text{if $0\leq \tau\leq 1$}\,,\\
& b_{t,\tau}(x,s,\xi) &&= 
T_{1/\tau}\left(b_t(x,s,\xi)\right)
&&\qquad\text{if $0<\tau\leq 1$}\,,\\
& b_{t,\tau}(x,s,\xi) &&=
b_t(x,s,\xi)
&&\qquad\text{if $\tau=0$}\,,
\end{alignat*}
for $(t,\tau)\in T=[0,1]\times[0,1]$.
It is easily seen that $a_{t,\tau}$ and $b_{t,\tau}$
satisfy $(u_1)$--$(u_3)$.
\par
Define also
\[
H_{t,\tau}(u) =
- \mathrm{div}[a_{t,\tau}(x,\nabla u)] 
+ b_{t,\tau}(x,u,\nabla u)\,,
\]
for $0\leq t\leq 1$, $0<\tau\leq 1$ and $u\in W^{1,p}_0(\Omega)$.
\par
Since $[0,1]\times\{0\}$ is compact, by~Corollary~\ref{cor:closed}
there exists $\overline{\tau}\in]0,1]$ such that 
\[
\begin{cases}
- \mathrm{div}[a_{t,\tau}(x,\nabla u)] 
+ b_{t,\tau}(x,u,\nabla u)=0 
&\qquad\text{in $\Omega$}\,,\\
u=0
&\qquad\text{on $\partial\Omega$}\,,
\end{cases}
\]
has no entropy solution with 
$(t,\tau)\in[0,1]\times[0,\overline{\tau}]$ and $u\in\partial U$.
\par
For any $t\in [0,1]$, it follows
\[
\mathrm{deg}(- \mathrm{div}[a_t(x,\nabla u)] 
+ b_t(x,u,\nabla u),U,0) =
\widetilde{\mathrm{deg}}_{(S)_+}(H_{t,\overline{\tau}},
U\cap W^{1,p}_0(\Omega),0)\,.
\]
On the other hand, as before, the map 
$\left\{(u,t)\mapsto H_{t,\overline{\tau}}(u)\right\}$
is continuous and of class~$(S)_+$ and the set
\[
\left\{(u,t)\in (U\cap W^{1,p}_0(\Omega))\times
[0,1]:\,\,H_{t,\overline{\tau}}(u)=0\right\}
\]
is compact in $W^{1,p}_0(\Omega)\times[0,1]$.
By Theorem~\ref{thm:homotopyc}
we conclude that 
\[
\widetilde{\mathrm{deg}}_{(S)_+}(H_{t,\overline{\tau}},
U\cap W^{1,p}_0(\Omega),0)
\]
is independent of $t$.
\end{proof}
\begin{cor}
\label{cor:dec}
Let $U$ be a bounded and open subset of $\Phi^{1,p}_0(\Omega)$
such that~\eqref{eq:b0} has no entropy solution  
$u\in\partial U$.
Let also
\[
\hat{a}:\Omega\times\R^N\rightarrow\R^N\,,\qquad
\hat{b}:\Omega\times(\R\times\R^N)\rightarrow\R
\]
be two Carath\'eodory functions satisfying $(a_1)$ -- $(a_3)$
and such that
\begin{multline*}
\int_{\Omega} \{a(x,\nabla z)\cdot\nabla v + b(x,z,\nabla z)v\}\,dx
= \int_{\Omega} \{\hat{a}(x,\nabla z)\cdot\nabla v 
+ \hat{b}(x,z,\nabla z)v\}\,dx \\
\qquad \forall z, v\in W^{1,p}_0(\Omega)\cap L^{\infty}(\Omega)\,.
\end{multline*}
\indent
Then 
\begin{multline*}
\mathrm{deg}(- \mathrm{div}[a(x,\nabla u)] 
+ b(x,u,\nabla u),U,0) \\
= 
\mathrm{deg}(- \mathrm{div}[\hat{a}(x,\nabla u)] 
+ \hat{b}(x,u,\nabla u),U,0)\,.
\end{multline*}
\end{cor}
\begin{proof}
Taking into account Remark~\ref{rem:dec}, it is enough
to apply Theorem~\ref{thm:homotopy0} to
\[
a_t(x,\xi) = (1-t)a(x,\xi) + t\hat{a}(x,\xi)\,,\quad
b_t(x,s,\xi) = (1-t)b(x,s,\xi) + t\hat{b}(x,s,\xi)\,.
\]
\end{proof}
Now we treat~\eqref{eq:bmu} in the general case.
According 
to~\cite{boccardo_gallouet_orsina1996},
any $\mu\in\mathcal{M}_b^p(\Omega)$ satisfies
\[
\int_\Omega v\,d\mu 
= \int_\Omega v w_0\,dx 
+\int_\Omega (\nabla v)\cdot w_1\,dx
\qquad\forall v\in W^{1,p}_0(\Omega)\cap L^\infty(\Omega)\,,
\]
for some $w_0\in L^1(\Omega)$ and 
$w_1\in L^{p'}(\Omega;\R^N)$.
On the other hand, if $w_0\in L^1(\Omega)$ and 
$w_1\in L^{p'}(\Omega;\R^N)$,
it is easily seen that the functions
\[
\tilde{a}(x,\xi) = a(x,\xi) - w_1(x)\,,\qquad
\tilde{b}(x,s,\xi) = b(x,s,\xi) - w_0(x)
\]
still satisfy $(a_1)$ -- $(a_3)$.
\begin{prop}
Let $\mu\in\mathcal{M}_b^p(\Omega)$ and let $U$ be a bounded 
and open subset of $\Phi^{1,p}_0(\Omega)$ such 
that~\eqref{eq:bmu}
has no entropy solution $u\in\partial U$.
Let also
$w_0, \hat{w}_0\in L^1(\Omega)$ and
$w_1, \hat{w}_1\in L^{p'}(\Omega;\R^N)$ be such that
\begin{multline*}
\int_\Omega v\,d\mu 
= \int_\Omega vw_0\,dx 
+\int_\Omega (\nabla v)\cdot w_1\,dx 
= \int_\Omega v\hat{w}_0\,dx 
+\int_\Omega (\nabla v)\cdot \hat{w}_1\,dx \\
\qquad\forall v\in W^{1,p}_0(\Omega)\cap L^\infty(\Omega)\,.
\end{multline*}
\indent
Then
\begin{multline*}
\mathrm{deg}(- \mathrm{div}[a(x,\nabla u)-w_1] 
+ b(x,u,\nabla u)-w_0,U,0) \\
= 
\mathrm{deg}(- \mathrm{div}[a(x,\nabla u)-\hat{w}_1] 
+ b(x,u,\nabla u) - \hat{w}_0,U,0)\,.
\end{multline*}
\end{prop}
\begin{proof}
It is a consequence of Corollary~\ref{cor:dec}.
\end{proof}
\begin{defn}
\label{defn:degree}
Let $\mu\in\mathcal{M}_b^p(\Omega)$ and let $U$ be a bounded 
and open subset of $\Phi^{1,p}_0(\Omega)$ such 
that~\eqref{eq:bmu}
has no entropy solution $u\in\partial U$.
\par
We set
\begin{multline*}
\mathrm{deg}(- \mathrm{div}[a(x,\nabla u)] 
+ b(x,u,\nabla u),U,\mu) \\
= \mathrm{deg}(- \mathrm{div}[a(x,\nabla u)-w_1] 
+ b(x,u,\nabla u)-w_0,U,0) \,,
\end{multline*}
where $w_0\in L^1(\Omega)$ and $w_1\in L^{p'}(\Omega;\R^N)$
satisfy
\[
\int_\Omega v\,d\mu 
= \int_\Omega vw_0\,dx 
+\int_\Omega (\nabla v)\cdot w_1\,dx 
\qquad\forall v\in W^{1,p}_0(\Omega)\cap L^\infty(\Omega)\,.
\]
\end{defn}
Theorems~\ref{thm:consistency}, ~\ref{thm:existence}, 
\ref{thm:additivity}, \ref{thm:excision} 
and~\ref{thm:homotopy} can be easily reduced to
Theorems~\ref{thm:consistency0}, ~\ref{thm:existence0}, 
\ref{thm:additivity0}, \ref{thm:excision0} 
and~\ref{thm:homotopy0}, respectively.
Let us see in detail the normalization property.
\par\smallskip\noindent
\emph{Proof of Theorem~\ref{thm:normalization}.}
\par\noindent
If we set
\[
a_t(x,\xi)=a(x,\xi) - (1-t) a(x,0)\,,
\quad b_t(x,s,\xi)=0\,,
\quad \mu_0=0\,,\quad \mu_1=\mu\,,
\]
it is easily seen that $a_t, b_t$ satisfy $(u_1)$ -- $(u_3)$,
possibly with $\nu$, $\alpha_0$, $\alpha_1$ and $\beta_1$ 
replaced by some $\hat{\nu}$, $\hat{\alpha}_0$, 
$\hat{\alpha}_1$ and $\hat{\beta}_1$.
If we set
\[
V = U\cup\left\{v\in \Phi^{1,p}_0(\Omega):\,\,
\hat{\nu}\|\nabla[\varphi_p(v)]\|_p <
\|\psi\|_\infty |\mu|(\Omega) 
+ \|\hat{\alpha}_0\|_1 +1\right\}\,,
\]
by~\eqref{eq:estimate} there is no entropy solution 
of~\eqref{eq:bmut} with $t\in[0,1]$ and $u\in\partial V$.
By Theorem~\ref{thm:homotopy}, it follows
\[
\mathrm{deg}(- \mathrm{div}[a(x,\nabla u)],V,\mu) =
\mathrm{deg}(- \mathrm{div}[a(x,\nabla u)
-a(x,0)],V,0) 
\]
and we also have
\[
\mathrm{deg}(- \mathrm{div}[a(x,\nabla u)],V,\mu) =
\mathrm{deg}(- \mathrm{div}[a(x,\nabla u)],U,\mu)
\]
by Theorem~\ref{thm:excision}.
\par
On the other hand, if $W$ is any bounded and open neighborhood 
of $0$ in $W^{1,p}_0(\Omega)$, by Theorem~\ref{thm:consistency} 
it holds
\begin{multline*}
\mathrm{deg}(- \mathrm{div}[a(x,\nabla u)-a(x,0)],V,0) \\
=
\mathrm{deg}_{(S)_+}(- \mathrm{div}[a(x,\nabla u)-a(x,0)],
V\cap W,0)
\end{multline*}
and, finally, 
\[
\mathrm{deg}_{(S)_+}(- \mathrm{div}[a(x,\nabla u)-a(x,0)],
V\cap W,0) = 1
\]
by Theorem~\ref{thm:normalizationw} with $u_0=0$.
\qed
%


%
\end{document}